\documentclass[12pt,twoside,reqno]{amsart}
\usepackage[colorlinks=false,citecolor=blue]{hyperref}
\usepackage{mathptmx, amsmath, amssymb, amsfonts, amsthm, mathptmx, enumerate, color,mathrsfs}
\usepackage{graphicx}
\newcommand{\vertiii}[1]{{\left\vert\kern-0.25ex\left\vert\kern-0.25ex\left\vert #1 
    \right\vert\kern-0.25ex\right\vert\kern-0.25ex\right\vert}}

\usepackage{epstopdf}
\newtheorem{theorem}{Theorem}[section]
\newtheorem{lemma}{Lemma}[section]
\newtheorem{remark}{Remark}[section]

\newtheorem{corollary}{Corollary}[section]

\newtheorem{proposition}{Proposition}[section]
\numberwithin{equation}{section}

\begin{document}
\title{Numerical radius bounds via the Euclidean operator radius and norm}
\author{Mohammad Sababheh and Hamid Reza Moradi}
\subjclass[2010]{Primary 47A12; Secondary 47A30, 47A63, 47B15}
\keywords{Euclidean operator radius, numerical radius, norm inequality, operator matrix}

\begin{abstract}
In this paper, we begin by showing a new generalization of the celebrated Cauchy-Schwarz inequality for the inner product. Then, this generalization is used to present some bounds for the Euclidean operator radius and the Euclidean operator norm.

These bounds will be used then to obtain some bounds for the numerical radius in a way that extends many well-known results in many cases.

The obtained results will be compared with the existing literature through numerical examples and rigorous approaches, whoever is applicable. In this context, more than 15 numerical examples will be given to support the advantage of our findings.

Among many consequences, will show that if $T$ is an accretive-dissipative bounded linear operator on a Hilbert space, then ${{\left\| \left( \Re T,\Im T \right) \right\|}_{e}}=\omega \left( T \right)$, where $\omega(\cdot), \|(\cdot,\cdot)\|_e, \Re T$ and $\Im T$ denote, respectively, the numerical radius, the Euclidean norm, the real part and the imaginary part.
\end{abstract}

\maketitle
\pagestyle{myheadings}
\markboth{\centerline {}}
{\centerline {}}
\bigskip
\bigskip

\section{Introduction}
Let $\mathbb{B}(\mathbb{H})$ be the $C^*$-algebra of all bounded linear operators on a complex Hilbert space $\mathbb{H}$, and let $\|\cdot\|$ be the operator norm defined by $\|T\|=\sup\limits_{\|x\|=1}\|Tx\|.$ In this context, if $x\in\mathbb{H}$, the quantity $\|x\|$ is defined by $\left<x,x\right>^{\frac{1}{2}},$ where $\left<\cdot,\cdot\right>$ is the inner product defined on $\mathbb{H}.$

An equivalent definition for the operator norm can be stated as $\|T\|=\sup\limits_{\|x\|=\|y\|=1}\left|\left<Tx,y\right>\right|.$ If, in this latter definition, $y=x$, a smaller quantity known as the numerical radius and denoted $\omega(T)$ is obtained. Thus, for $T\in\mathbb{B}(\mathbb{H})$, the numerical radius of $T$ is the scalar quantity $\omega(T)=\sup\limits_{\|x\|=1}\left|\left<Tx,x\right>\right|.$ It is easily verified that $\omega(\cdot)$ defines a norm on $\mathbb{B}(\mathbb{H})$, as well. However, there are major differences between the norm properties of $\omega(\cdot)$ and $\|\cdot\|$. More precisely, the numerical radius is not sub-multiplicative nor unitarily invariant, unlike the operator norm.

Although the definition of $\omega(\cdot)$ seems easier than that of $\|\cdot\|$, the calculations of $\omega(\cdot)$ turn out to be much more complicated. Thus, it has been a key interest in the literature to find some approximate values of $\omega(\cdot)$ in terms of $\|\cdot\|.$ This is usually done through some sharp upper and lower bounds.

In this direction, the relation \cite[Theorem 1.3-1]{Gustafson_Book_1997}
\begin{equation}\label{eq_intro_equiv}
\frac{1}{2}\|T\|\leq\omega(T)\leq\|T\|
\end{equation}
is a basic relation that furnishes the equivalence of the two norms $\omega(\cdot)$ and $\|\cdot\|$.

As the difference between the left and right sides of \eqref{eq_intro_equiv} can be very large, researchers have devoted a considerable effort to finding tighter bounds that could be used for approximation targets or to give new insight into such relations. 

Among the most useful and simple upper bounds in this context, we have
\cite[Ineq. (8)]{1}
\begin{equation}\label{eq_kitt_1}
\omega(T)\leq\frac{1}{2}\left\| \;|T|+|T^*|\;\right\|,
\end{equation}
and
\cite[Theorem 1]{7}
\begin{equation}\label{eq_kitt_2}
\omega^2(T)\leq\frac{1}{2}\left\| |T|^2+|T^*|^2\right\|.
\end{equation}

In fact, using convexity of the function $f(t)=t^2$, it can be seen that \eqref{eq_kitt_1} is sharper than \eqref{eq_kitt_2}. We refer the reader to \cite{MFS, SMF} where this concern is discussed.

In \cite{11}, the notions of the numerical radius and  the operator norm were extended  
for $n$-tuple of operators $\left( {{T}_{1}},{{T}_{2}},\ldots ,{{T}_{n}} \right)$ by the Euclidean operator radius, defined by
\[{{\omega }_{e}}\left( {{T}_{1}},{{T}_{2}},\ldots ,{{T}_{n}} \right)=\underset{\left\| x \right\|=1}{\mathop{\sup }}\,{{\left( \sum\limits_{i=1}^{n}{{{\left| \left\langle {{T}_{i}}x,x \right\rangle  \right|}^{2}}} \right)}^{\frac{1}{2}}},\]
and the Euclidean operator norm, defined by
\[{{\left\| \left( {{T}_{1}},{{T}_{2}},\ldots ,{{T}_{n}} \right) \right\|}_{e}}=\underset{\left( {{\lambda }_{1}},{{\lambda }_{2}},\ldots ,{{\lambda }_{n}} \right)\in {{\mathbb B}_{n}}}{\mathop{\sup }}\,\left\| {{\lambda }_{1}}{{T}_{1}}+{{\lambda }_{2}}{{T}_{2}}+\cdots +{{\lambda }_{n}}{{T}_{n}} \right\|\]
where ${{\mathbb B}_{n}}=\left\{ \left( {{\lambda }_{1}},{{\lambda }_{2}},\ldots ,{{\lambda }_{n}} \right)\in {{\mathbb{C}}^{n}}:\text{ }\sum\nolimits_{i=1}^{n}{{{\left| {{\lambda }_{i}} \right|}^{2}}}\le 1 \right\}$.

It has been shown in \cite[Theorem 6.1]{10} that
\[{{\left\| \left( {{T}_{1}},{{T}_{2}},\ldots ,{{T}_{n}} \right) \right\|}_{e}}=\underset{\left\| x \right\|=\left\| y \right\|=1}{\mathop{\sup }}\,{{\left( \sum\limits_{i=1}^{n}{{{\left| \left\langle {{T}_{i}}x,y \right\rangle  \right|}^{2}}} \right)}^{\frac{1}{2}}},\]
as an alternative formula for $\|\cdot\|_e.$ Both $\omega_e(\cdot)$ and $\|\cdot\|_e$ have been investigated in the literature, as one can see in \cite{Alomari_Georgian_2023, Dragomir_LAA_2006, Moslehian_MathScand_2017}.

One of the most basic yet useful tools for obtaining possible bounds for the numerical radius and the operator norm is the celebrated Cauchy-Schwarz inequality, which states $\left|\left<x,y\right>\right|\leq\|x\|\;\|y\|$, for $x,y\in\mathbb{H}.$ In this paper, as a first contribution, we prove a new generalized form of this inequality that allows obtaining some new and simple relations among $\omega(\cdot), \|\cdot\|, \omega_e(\cdot)$ and $\|\cdot\|_e.$

More precisely, we will show that if $x,y,z\in \mathbb H$, then
\begin{equation*}
{{\left| \left\langle x,y \right\rangle  \right|}^{2}}+{{\left| \left\langle x,z \right\rangle  \right|}^{2}}\le \left\| x \right\|\sqrt{{{\left| \left\langle x,y \right\rangle  \right|}^{2}}{{\left\| y \right\|}^{2}}+{{\left| \left\langle x,z \right\rangle  \right|}^{2}}{{\left\| z \right\|}^{2}}+2\left| \left\langle x,y \right\rangle  \right|\left| \left\langle x,z \right\rangle  \right|\left| \left\langle y,z \right\rangle  \right|}.
\end{equation*}

Many other bounds and forms will be shown and compared with the existing literature. Then, the above inequality will be utilized to obtain some bounds on the numerical radius and the Euclidean operator radius. 

In particular, we find an upper bound of $\omega_e(\cdot,\cdot)$ in terms of $\omega(\cdot)$ (see Theorem \ref{2}), and we compare this bound with the existing bounds. Then, this bound is used to obtain a known bound for $\omega(\cdot)$ (see Corollary \ref{8}). A more elaborated bound for $\omega_e(\cdot)$ will be found in Theorem \ref{7}, with an application to $\omega(\cdot)$ in Corollary \ref{9}. This latter corollary will be utilized to find an upper bound of $\omega\left(\left[\begin{matrix}O&A\\B^*&O\end{matrix}\right]\right)$ that can be better than previously known bounds, as we see in Corollary \ref{19} and Remark \ref{rem_main_block}. Similar discussion of $\|(\cdot,\cdot)\|_e$ will be also presented.

Among many other results, the following results are of interest for $A, B,T\in\mathbb{B}(\mathbb{H})$.
\[\omega _{e}^{2}\left( A,B \right)\le \sqrt{\left\| {{\omega }^{2}}\left( A \right){{\left| {{A}^{*}} \right|}^{2}}+{{\omega }^{2}}\left( B \right){{\left| {{B}^{*}} \right|}^{2}} \right\|+2\omega \left( A \right)\omega \left( B \right)\omega \left( A{{B}^{*}} \right)},\]

\[\left\| \left( A,B \right) \right\|_{e}^{2}\le \sqrt{\omega \left( {{\left| A \right|}^{2}}+{\rm i}{{\left| B \right|}^{2}} \right)\omega \left( {{\left| {{A}^{*}} \right|}^{2}}+{\rm i}{{\left| {{B}^{*}} \right|}^{2}} \right)+\omega \left( \left| A \right|+{\rm i}\left| B \right| \right)\omega \left( \left| {{A}^{*}} \right|+{\rm i}\left| {{B}^{*}} \right| \right)\omega \left( B{{A}^{*}} \right)},\]

\[
\omega \left( T \right)\le \frac{1}{2}\sqrt{\sqrt{2}\omega \left( {{\left| T \right|}^{2}}+{\rm i}{{\left| {{T}^{*}} \right|}^{2}} \right)+2\omega \left( {{T}^{2}} \right)},
\]
and
\[\begin{aligned}
   \frac{1}{2}\left\| T \right\|&\le \omega \left( \left[ \begin{matrix}
   O & \Re T  \\
   \Im T & O  \\
\end{matrix} \right] \right) \\ 
 & \le \frac{\sqrt{2}}{2}{{\left\| \left( \Re T,\Im T \right) \right\|}_{e}} \\ 
 & \le \frac{\sqrt{2}}{2}\omega \left( \left| \Re T \right|+{\rm i}\left| \Re T \right| \right) \\ 
 & \le \frac{\sqrt{2}}{2}{{\left\| {{\left( \Re T \right)}^{2}}+{{\left( \Im T \right)}^{2}} \right\|}^{\frac{1}{2}}} \\ 
 & \le \frac{\sqrt{2}}{2}\sqrt{{{\left\| \Re T \right\|}^{2}}+{{\left\| \Im T \right\|}^{2}}} \\ 
 & \le \omega \left( T \right).  
\end{aligned}\]

As mentioned earlier, the significance of the results will be explained through a sequence of remarks with numerous numerical examples.

Before proceeding to the main results, we recall the reader's attention to some lemmas. The first lemma treats operator matrices, which are operators defined on $\mathbb{H}\oplus \mathbb{H}$. Such operators are usually described by writing $\left[\begin{matrix}A&B\\C&D\end{matrix}\right]$, where $A,B,C,D\in\mathbb{B}(\mathbb{H})$.

\begin{lemma}\label{lem_hirz}
\cite[(4.6)]{8} Let $A,B\in \mathbb B(\mathbb{H})$. Then
\[\omega\left( \left[\begin{matrix}O&A\\B^*&O\end{matrix}\right]     \right)=\frac{1}{2}\sup_{\theta\in\mathbb{R}}\|A+e^{{\rm i}\theta}B\|,\]
where $O$ is the zero operator in $\mathbb{B}(\mathbb{H})$.
In particular (see \cite[Theorem 2.3]{8}),
\[\omega\left( \left[\begin{matrix}O&A\\B^*&O\end{matrix}\right]     \right)\leq\frac{\|A\|+\|B\|}{2}.\]
\end{lemma}

\begin{lemma}
(mixed Schwarz inequality \cite[pp. 75--76]{9}) Let $T\in \mathbb B(\mathbb{H})$. Then for any $x,y \in \mathbb H$,
\[\left| \left\langle Tx,y \right\rangle  \right|\le \sqrt{\left\langle \left| T \right|x,x \right\rangle \left\langle \left| {{T}^{*}} \right|y,y \right\rangle }.\]
\end{lemma}

We also have the following simple observation. For any $A,B\in \mathbb B\left( \mathbb H \right)$,
\begin{equation}\label{3}
{{\omega }_{e}}\left( A,B \right)\le \sqrt{\omega \left( A \right)\left\| A \right\|+\omega \left( B \right)\left\| B \right\|},
\end{equation}
holds. Indeed, 
\[\begin{aligned}
   \sqrt{{{\left| \left\langle Ax,x \right\rangle  \right|}^{2}}+{{\left| \left\langle Bx,x \right\rangle  \right|}^{2}}}&\le \sqrt{{{\omega }^{2}}\left( A \right)+{{\omega }^{2}}\left( B \right)} \\ 
 & \le \sqrt{\omega \left( A \right)\left\| A \right\|+\omega \left( B \right)\left\| B \right\|}.  
\end{aligned}\]
Now by taking supremum over $x\in \mathbb H$ with $\left\| x \right\|=1$ we get \eqref{3}. 

From \cite[Corollary 1]{1}, we know that if ${{X}^{2}}=O$, then $\omega \left( X \right)= \frac{1}{2}\left\| X \right\|$. So, if ${{A}^{2}}={{B}^{2}}=O$, then from \eqref{3}, we infer that
\[{{\omega }_{e}}\left( A,B \right)\le \sqrt{\frac{{{\left\| A \right\|}^{2}}+{{\left\| B \right\|}^{2}}}{2}}.\]

\section{Main Results}
In this section, we present our main results. For organizational purposes, we present our results in four subsections. In the first subsection, we prove the generalized form of the Cauchy-Schwarz inequality with some simple consequences; then, we discuss possible relations for $\omega_e(\cdot,\cdot)$. After that the quantity $\|(\cdot,\cdot)\|_e$ is discussed, and applications towards $\omega(\cdot)$ are presented in the end. Many numerical examples will be given to show the result's significance compared to the existing literature.
 \subsection{The generalized Cauchy-Schwarz inequality}
 Now, we show the following generalization of the Cauchy-Schwarz inequality.
\begin{lemma}\label{1}
Let $x,y,z\in \mathbb H$. Then 
\begin{equation}\label{5}
{{\left| \left\langle x,y \right\rangle  \right|}^{2}}+{{\left| \left\langle x,z \right\rangle  \right|}^{2}}\le \left\| x \right\|\sqrt{{{\left| \left\langle x,y \right\rangle  \right|}^{2}}{{\left\| y \right\|}^{2}}+{{\left| \left\langle x,z \right\rangle  \right|}^{2}}{{\left\| z \right\|}^{2}}+2\left| \left\langle x,y \right\rangle  \right|\left| \left\langle x,z \right\rangle  \right|\left| \left\langle y,z \right\rangle  \right|}.
\end{equation}
In particular,
\begin{equation}\label{4}
{{\left| \left\langle x,e \right\rangle  \right|}^{2}}+{{\left| \left\langle y,e \right\rangle  \right|}^{2}}\le \sqrt{{{\left| \left\langle x,e \right\rangle  \right|}^{2}}{{\left\| x \right\|}^{2}}+{{\left| \left\langle y,e \right\rangle  \right|}^{2}}{{\left\| y \right\|}^{2}}+2\left| \left\langle x,e \right\rangle  \right|\left| \left\langle y,e \right\rangle  \right|\left| \left\langle x,y \right\rangle  \right|},
\end{equation}
provided that $e\in \mathbb H$ with $\left\| e \right\|=1$. 
\end{lemma}
\begin{proof}
We have
\begin{equation}\label{25}
\begin{aligned}
  & {{\left( {{\left| \left\langle x,y \right\rangle  \right|}^{2}}+{{\left| \left\langle x,z \right\rangle  \right|}^{2}} \right)}^{2}} \\ 
 & ={{\left( \left\langle x,y \right\rangle \left\langle y,x \right\rangle +\left\langle x,z \right\rangle \left\langle z,x \right\rangle  \right)}^{2}} \\ 
 & ={{\left\langle x,\left\langle x,y \right\rangle y+\left\langle x,z \right\rangle z \right\rangle }^{2}} \\ 
 & \le {{\left\| x \right\|}^{2}}{{\left\| \left\langle x,y \right\rangle y+\left\langle x,z \right\rangle z \right\|}^{2}} \\ 
 &\quad \text{(by the Cauchy-Schwarz inequality)}\\
 & ={{\left\| x \right\|}^{2}}\left( {{\left| \left\langle x,y \right\rangle  \right|}^{2}}{{\left\| y \right\|}^{2}}+{{\left| \left\langle x,z \right\rangle  \right|}^{2}}{{\left\| z \right\|}^{2}}+2\Re\left( \left\langle x,y \right\rangle \overline{\left\langle x,z \right\rangle }\left\langle y,z \right\rangle  \right) \right) \\ 
 & \le {{\left\| x \right\|}^{2}}\left({{\left| \left\langle x,y \right\rangle  \right|}^{2}}{{\left\| y \right\|}^{2}}+{{\left| \left\langle x,z \right\rangle  \right|}^{2}}{{\left\| z \right\|}^{2}}+2\left| \left\langle x,y \right\rangle  \right|\left| \left\langle x,z \right\rangle  \right|\left| \left\langle y,z \right\rangle  \right|\right),\\
\end{aligned}
\end{equation}
where the we have used the fact that  $\Re a\le \left| a \right|$ for any $a\in \mathbb{C}$
to obtain the last inequality. This proves the first desired inequality.
The second inequality is observed from the first inequality by replacing $x$ with $e$ and substituting $y$ and $z$ with $x$  and $y$, respectively.
\end{proof}
In the following remark, we see how Lemma \ref{1} generalizes the Cauchy-Schwarz inequality.
\begin{remark}
\hfill 
\begin{itemize}
\item[(i)] If $y=z$, then
\[\left| \left\langle x,y \right\rangle  \right|\le \left\| x \right\|\left\| y \right\|.\]
\item[(ii)] If $y\bot z$, then
\[{{\left| \left\langle x,y \right\rangle  \right|}^{2}}+{{\left| \left\langle x,z \right\rangle  \right|}^{2}}\le \left\| x \right\|\sqrt{{{\left| \left\langle x,y \right\rangle  \right|}^{2}}{{\left\| y \right\|}^{2}}+{{\left| \left\langle x,z \right\rangle  \right|}^{2}}{{\left\| z \right\|}^{2}}}.\]
\end{itemize}

\end{remark}

\begin{remark}
It observes from the first inequality in \eqref{25} that
\begin{equation}\label{26}
{{\left| \left\langle x,y \right\rangle  \right|}^{2}}+{{\left| \left\langle x,z \right\rangle  \right|}^{2}}\le \left\| x \right\|\left\| \left\langle x,y \right\rangle y+\left\langle x,z \right\rangle z \right\|.
\end{equation}
In particular,
\begin{equation}\label{27}
{{\left| \left\langle x,e \right\rangle  \right|}^{2}}+{{\left| \left\langle y,e \right\rangle  \right|}^{2}}\le \left\| \left\langle e,x \right\rangle x+\left\langle e,y \right\rangle y \right\|.
\end{equation}
\end{remark}

For any $0\ne a,b\in \mathbb H$, one can define the angle between $a,b$ by the formula
\begin{equation}\label{6}
\cos {{\Psi }_{ab}}=\frac{\left| \left\langle a,b \right\rangle  \right|}{\left\| a \right\|\left\| b \right\|},\; 0\leq \Psi_{ab}\leq\frac{\pi}{2}.
\end{equation}
We refer the reader to \cite{Krein_FAA_1969, Lin_Intelligencer_2012, Sababheh_OaM_2022, Scharnhorst_ActaApplMath_2001} for discussion of this definition and another definition for the angle. In the following result, we prove the following additional property.
\begin{corollary}
Let $x,y,z\in \mathbb H$ be nonzero. Then 
\[\cos {{\Psi }_{yx}}\;\cos {{\Psi }_{xz}}\le \frac{1}{2}\sqrt{\cos^{2} \Psi _{yx}+\cos^{2} \Psi _{xz}+2\cos {{\Psi }_{yx}}\;\cos {{\Psi }_{xz}}\;\cos {{\Psi }_{zy}}}.\]
\end{corollary}
\begin{proof}
If we utilize the arithmetic-geometric mean inequality in the left-side of inequality \eqref{5}, we conclude that
\[2\left| \left\langle x,y \right\rangle  \right|\left| \left\langle x,z \right\rangle  \right|\le \left\| x \right\|\sqrt{{{\left| \left\langle x,y \right\rangle  \right|}^{2}}{{\left\| y \right\|}^{2}}+{{\left| \left\langle x,z \right\rangle  \right|}^{2}}{{\left\| z \right\|}^{2}}+2\left| \left\langle x,y \right\rangle  \right|\left| \left\langle x,z \right\rangle  \right|\left| \left\langle y,z \right\rangle  \right|}.\]
Replacing $x,y,z$ by  $\frac{x}{\left\| x \right\|},\frac{y}{\left\| y \right\|},\frac{z}{\left\| z \right\|}$, we have
\[2\frac{\left| \left\langle x,y \right\rangle  \right|\left| \left\langle x,z \right\rangle  \right|}{{{\left\| x \right\|}^{2}}\left\| y \right\|\left\| z \right\|}\le \sqrt{\frac{{{\left| \left\langle x,y \right\rangle  \right|}^{2}}}{{{\left\| x \right\|}^{2}}{{\left\| y \right\|}^{2}}}+\frac{{{\left| \left\langle x,z \right\rangle  \right|}^{2}}}{{{\left\| x \right\|}^{2}}{{\left\| z \right\|}^{2}}}+2\frac{\left| \left\langle x,y \right\rangle  \right|\left| \left\langle x,z \right\rangle  \right|\left| \left\langle y,z \right\rangle  \right|}{{{\left\| x \right\|}^{2}}{{\left\| y \right\|}^{2}}{{\left\| z \right\|}^{2}}}},\]
which is equivalent to the desired result, thanks to \eqref{6}.
\end{proof}

\begin{remark}
Letting $x=Tx$, $y={{T}^{*}}x$, and $e=x$ with $\left\| x \right\|=1$, in \eqref{4}, we reach the following well-known inequality (see \cite[Ineq. (3.15)]{2})
\[\left| \left\langle Tx,x \right\rangle  \right|\le \frac{1}{2}\sqrt{\left\langle \left( {{\left| T \right|}^{2}}+{{\left| {{T}^{*}} \right|}^{2}} \right)x,x \right\rangle +2\left| \left\langle {{T}^{2}}x,x \right\rangle  \right|}.\]
In particular, 
\begin{equation}\label{29}
\omega^2(T)\leq\frac{1}{4}\left\| |T|^2+|T^*|^2\right\|+\frac{1}{2}\omega(T^2).
\end{equation}
Indeed,
\[\begin{aligned}
  & {{\left| \left\langle Tx,x \right\rangle  \right|}^{2}} \\ 
 & \le \frac{1}{2}\sqrt{{{\left| \left\langle Tx,x \right\rangle  \right|}^{2}}\left( {{\left\| Tx \right\|}^{2}}+{{\left\| {{T}^{*}}x \right\|}^{2}} \right)+2{{\left| \left\langle Tx,x \right\rangle  \right|}^{2}}\left| \left\langle {{T}^{2}}x,x \right\rangle  \right|} \\ 
 & =\frac{1}{2}\left| \left\langle Tx,x \right\rangle  \right|\sqrt{\left\langle \left( {{\left| T \right|}^{2}}+{{\left| {{T}^{*}} \right|}^{2}} \right)x,x \right\rangle +2\left| \left\langle {{T}^{2}}x,x \right\rangle  \right|}. 
\end{aligned}\]
\end{remark}

\subsection{Upper bounds for $\omega_e(\cdot,\cdot)$}
In this subsection, we present some bounds for the Euclidean operator radius. The applications of these bounds towards the numerical radius will be given in Subsection \ref{section_w}. Although simpler bounds are known in the literature in terms of $\|\cdot\|$, the following bound involves the smaller quantity $\omega(\cdot).$
\begin{theorem}\label{2}
Let $A,B\in \mathbb B\left( \mathbb H \right)$. Then
\[\omega _{e}^{2}\left( A,B \right)\le \sqrt{{{\omega }^{2}}\left( A \right){{\left\| A \right\|}^{2}}+{{\omega }^{2}}\left( B \right){{\left\| B \right\|}^{2}}+2\omega \left( A \right)\omega \left( B \right)\omega \left( {{B}^{*}}A \right)}.\]
\end{theorem}
\begin{proof}
By substituting $x=Ax$, $y=Bx$, and $e=x$,  in \eqref{4}, we obtain
\[\begin{aligned}
  & {{\left| \left\langle Ax,x \right\rangle  \right|}^{2}}+{{\left| \left\langle Bx,x \right\rangle  \right|}^{2}} \\ 
 & \le \sqrt{{{\left| \left\langle Ax,x \right\rangle  \right|}^{2}}{{\left\| Ax \right\|}^{2}}+{{\left| \left\langle Bx,x \right\rangle  \right|}^{2}}{{\left\| Bx \right\|}^{2}}+2\left| \left\langle Ax,x \right\rangle  \right|\left| \left\langle Bx,x \right\rangle  \right|\left| \left\langle {{B}^{*}}Ax,x \right\rangle  \right|} \\ 
 & \le \sqrt{{{\omega }^{2}}\left( A \right){{\left\| A \right\|}^{2}}+{{\omega }^{2}}\left( B \right){{\left\| B \right\|}^{2}}+2\omega \left( A \right)\omega \left( B \right)\omega \left( {{B}^{*}}A \right)}, 
\end{aligned}\]
i.e.,
\begin{equation*}\label{10}
{{\left| \left\langle Ax,x \right\rangle  \right|}^{2}}+{{\left| \left\langle Bx,x \right\rangle  \right|}^{2}}\le \sqrt{{{\omega }^{2}}\left( A \right){{\left\| A \right\|}^{2}}+{{\omega }^{2}}\left( B \right){{\left\| B \right\|}^{2}}+2\omega \left( A \right)\omega \left( B \right)\omega \left( {{B}^{*}}A \right)}.
\end{equation*}
We obtain the desired result by taking supremum over all unit vectors $x\in \mathbb H$.
\end{proof}

\begin{remark}\label{rem_main_pop}
It has been shown in \cite{11} that
\begin{equation}\label{eq_rem_pop}
\omega_e^2(A,B)\leq \left\| |A|^2+|B|^2\right\|.
\end{equation}
This remark shows that Theorem \ref{2} can be better than this latter bound. Indeed, if we let  $A=\left[
\begin{array}{cc}
 2 & 3 \\
 1 & 0 \\
\end{array}
\right]$ and $B=\left[
\begin{array}{cc}
 2 & 2 \\
 5 & 3 \\
\end{array}
\right],$ then numerical calculations show that
\[\sqrt{{{\omega }^{2}}\left( A \right){{\left\| A \right\|}^{2}}+{{\omega }^{2}}\left( B \right){{\left\| B \right\|}^{2}}+2\omega \left( A \right)\omega \left( B \right)\omega \left( {{B}^{*}}A \right)}\approx 47.0005,\]
and
\[\left\| |A|^2+|B|^2\right\|\approx 53.7099.\]
This shows that, in this example, the bound we found in Theorem \ref{2} is better than that in \eqref{eq_rem_pop}.

However, if we take $A=\left[
\begin{array}{cc}
 4 & 0 \\
 1 & 3 \\
\end{array}
\right]$ and $B=\left[
\begin{array}{cc}
 1 & 3 \\
 0 & 5 \\
\end{array}
\right],$ we find that
\[\sqrt{{{\omega }^{2}}\left( A \right){{\left\| A \right\|}^{2}}+{{\omega }^{2}}\left( B \right){{\left\| B \right\|}^{2}}+2\omega \left( A \right)\omega \left( B \right)\omega \left( {{B}^{*}}A \right)}\approx 47.5757,\]
and
\[\left\| |A|^2+|B|^2\right\|\approx 44.3654.\]
Thus, neither Theorem \ref{2} nor \eqref{eq_rem_pop} is uniformly better than the other.
\end{remark}

Another application of \eqref{27} is the following tighter bound than the one given in Theorem \ref{2}. Since we have emphasized the significance of Theorem \ref{2} in Remark \ref{rem_main_pop}, the significance of Theorem \ref{28} is evident.
\begin{theorem}\label{28}
Let $A,B\in \mathbb B\left( \mathbb H \right)$. Then
	\[\omega _{e}^{2}\left( A,B \right)\le \sqrt{\left\| {{\omega }^{2}}\left( A \right){{\left| {{A}^{*}} \right|}^{2}}+{{\omega }^{2}}\left( B \right){{\left| {{B}^{*}} \right|}^{2}} \right\|+2\omega \left( A \right)\omega \left( B \right)\omega \left( A{{B}^{*}} \right)}.\]
\end{theorem}
\begin{proof}
Let $x\in\mathbb{H}$ be a unit vector. If we replace $x$ by $Ax$, $y$ by $Bx$ and $e$ by $x$  in \eqref{27}, we get
	\[\begin{aligned}
  & {{\left| \left\langle Ax,x \right\rangle  \right|}^{2}}+{{\left| \left\langle Bx,x \right\rangle  \right|}^{2}} \\ 
 & \le \left\| \left( \left\langle {{A}^{*}}x,x \right\rangle A+\left\langle {{B}^{*}}x,x \right\rangle B \right)x \right\| \\ 
 & \le \left\| \left\langle {{A}^{*}}x,x \right\rangle A+\left\langle {{B}^{*}}x,x \right\rangle B \right\| \\ 
 & ={{\left\| \left( \left\langle {{A}^{*}}x,x \right\rangle A+\left\langle {{B}^{*}}x,x \right\rangle B \right)\left( \left\langle Ax,x \right\rangle {{A}^{*}}+\left\langle Bx,x \right\rangle {{B}^{*}} \right) \right\|}^{\frac{1}{2}}} \\ 
 &\quad \text{(since $\left\| X{{X}^{*}} \right\|={{\left\| X \right\|}^{2}}$ for any $X\in \mathbb B\left( \mathbb H \right)$)}\\
 & ={{\left\| {{\left| \left\langle Ax,x \right\rangle  \right|}^{2}}{{\left| {{A}^{*}} \right|}^{2}}+{{\left| \left\langle Bx,x \right\rangle  \right|}^{2}}{{\left| {{B}^{*}} \right|}^{2}}+2\Re\left( \left\langle {{A}^{*}}x,x \right\rangle \left\langle Bx,x \right\rangle A{{B}^{*}} \right) \right\|}^{\frac{1}{2}}} \\ 
 & \le \sqrt{\left\| {{\left| \left\langle Ax,x \right\rangle  \right|}^{2}}{{\left| {{A}^{*}} \right|}^{2}}+{{\left| \left\langle Bx,x \right\rangle  \right|}^{2}}{{\left| {{B}^{*}} \right|}^{2}} \right\|+2\left\| \Re\left( \left\langle {{A}^{*}}x,x \right\rangle \left\langle Bx,x \right\rangle A{{B}^{*}} \right) \right\|} \\ 
 &\quad \text{(by the triangle inequality for the usual operator norm)}\\
 & \le \sqrt{\left\| {{\left| \left\langle Ax,x \right\rangle  \right|}^{2}}{{\left| {{A}^{*}} \right|}^{2}}+{{\left| \left\langle Bx,x \right\rangle  \right|}^{2}}{{\left| {{B}^{*}} \right|}^{2}} \right\|+2\left| \left\langle {{A}^{*}}x,x \right\rangle \left\langle Bx,x \right\rangle  \right|\omega \left( A{{B}^{*}} \right)} \\ 
 &\quad \text{(since $\left\| \Re X \right\|\le \omega \left( X \right)$ for any $X\in \mathbb B\left( \mathbb H \right)$)}\\
 & \le \sqrt{\left\| {{\omega }^{2}}\left( A \right){{\left| {{A}^{*}} \right|}^{2}}+{{\omega }^{2}}\left( B \right){{\left| {{B}^{*}} \right|}^{2}} \right\|+2\omega \left( A \right)\omega \left( B \right)\omega \left( A{{B}^{*}} \right)}.  
\end{aligned}\]
Consequently,
	\[{{\left| \left\langle Ax,x \right\rangle  \right|}^{2}}+{{\left| \left\langle Bx,x \right\rangle  \right|}^{2}}\le \sqrt{\left\| {{\omega }^{2}}\left( A \right){{\left| {{A}^{*}} \right|}^{2}}+{{\omega }^{2}}\left( B \right){{\left| {{B}^{*}} \right|}^{2}} \right\|+2\omega \left( A \right)\omega \left( B \right)\omega \left( A{{B}^{*}} \right)},\]
which indicates the desired inequality after taking supremum over all unit vectors $x\in \mathbb H$.
\end{proof}

We remark here that Theorem \ref{28} recovers \eqref{29} by substituting $A=T$ and $B=T^{*}$.

\begin{remark}
Since Theorem \ref{28} is a refinement of Theorem \ref{2}, Remark \ref{rem_main_pop} already explains the advantage of Theorem \ref{28} over \eqref{eq_rem_pop}. Here, we give an example to show that \eqref{eq_rem_pop} can also be better than Theorem \ref{28}. Indeed, if we take 
$A=\left[
\begin{array}{cc}
 4 & 3 \\
 4 & 2 \\
\end{array}
\right]$ and $B=\left[
\begin{array}{cc}
 0 & 4 \\
 3 & 0\\
\end{array}
\right],$ we find that
\[\sqrt{\left\| {{\omega }^{2}}\left( A \right){{\left| {{A}^{*}} \right|}^{2}}+{{\omega }^{2}}\left( B \right){{\left| {{B}^{*}} \right|}^{2}} \right\|+2\omega \left( A \right)\omega \left( B \right)\omega \left( A{{B}^{*}} \right)}\approx 56.1224,\]
and
\[\left\| |A|^2+|B|^2\right\|\approx 55.8806.\]
Thus, Theorem \ref{28} and \eqref{eq_rem_pop} are generally not comparable.
\end{remark}
Next, a new form involving the Cartesian decomposition as an upper bound of $\omega_e(\cdot,\cdot)$ is stated.
\begin{theorem}\label{7}
Let $A,B\in \mathbb B\left( \mathbb H \right)$. Then
\[\omega _{e}^{2}\left( A,B \right)\le \sqrt{\sqrt{\left( {{\omega }^{4}}\left( A \right)+{{\omega }^{4}}\left( B \right) \right)}\;\omega \left( {{\left| A \right|}^{2}}+{\rm i}{{\left| B \right|}^{2}} \right)+2\omega \left( A \right)\omega \left( B \right)\omega \left( {{B}^{*}}A \right)}.\]
\end{theorem}
\begin{proof}
By substituting $x=Ax$, $y=Bx$, and $e=x$,  in \eqref{4}, we obtain
\[\begin{aligned}
  & {{\left| \left\langle Ax,x \right\rangle  \right|}^{2}}+{{\left| \left\langle Bx,x \right\rangle  \right|}^{2}} \\ 
 & \le \sqrt{{{\left| \left\langle Ax,x \right\rangle  \right|}^{2}}{{\left\| Ax \right\|}^{2}}+{{\left| \left\langle Bx,x \right\rangle  \right|}^{2}}{{\left\| Bx \right\|}^{2}}+2\left| \left\langle Ax,x \right\rangle  \right|\left| \left\langle Bx,x \right\rangle  \right|\left| \left\langle {{B}^{*}}Ax,x \right\rangle  \right|} \\ 
 & \le \sqrt{\sqrt{\left( {{\left| \left\langle Ax,x \right\rangle  \right|}^{4}}+{{\left| \left\langle Bx,x \right\rangle  \right|}^{4}} \right)\left( {{\left\| Ax \right\|}^{4}}+{{\left\| Bx \right\|}^{4}} \right)}+2\left| \left\langle Ax,x \right\rangle  \right|\left| \left\langle Bx,x \right\rangle  \right|\left| \left\langle {{B}^{*}}Ax,x \right\rangle  \right|} \\ 
 &\quad \text{(by the Cauchy-Schwarz inequality)}\\
 & =\sqrt{\sqrt{\left( {{\left| \left\langle Ax,x \right\rangle  \right|}^{4}}+{{\left| \left\langle Bx,x \right\rangle  \right|}^{4}} \right)\left( {{\left\langle {{\left| A \right|}^{2}}x,x \right\rangle }^{2}}+{{\left\langle {{\left| B \right|}^{2}}x,x \right\rangle }^{2}} \right)}+2\left| \left\langle Ax,x \right\rangle  \right|\left| \left\langle Bx,x \right\rangle  \right|\left| \left\langle {{B}^{*}}Ax,x \right\rangle  \right|} \\ 
 & =\sqrt{\sqrt{ {{\left| \left\langle Ax,x \right\rangle  \right|}^{4}}+{{\left| \left\langle Bx,x \right\rangle  \right|}^{4}} }\;\left| \left\langle \left( {{\left| A \right|}^{2}}+{\rm i}{{\left| B \right|}^{2}} \right)x,x \right\rangle  \right|+2\left| \left\langle Ax,x \right\rangle  \right|\left| \left\langle Bx,x \right\rangle  \right|\left| \left\langle {{B}^{*}}Ax,x \right\rangle  \right|} \\ 
  &\quad \text{(since $\left| a+{\rm i}b \right|=\sqrt{{{a}^{2}}+{{b}^{2}}}$ for any $a,b\in \mathbb{R}$)}\\
 & \le \sqrt{\sqrt{{{\omega }^{4}}\left( A \right)+{{\omega }^{4}}\left( B \right)}\;\omega \left( {{\left| A \right|}^{2}}+{\rm i}{{\left| B \right|}^{2}} \right)+2\omega \left( A \right)\omega \left( B \right)\omega \left( {{B}^{*}}A \right)},  
\end{aligned}\]
i.e.,
\begin{equation}\label{12}
{{\left| \left\langle Ax,x \right\rangle  \right|}^{2}}+{{\left| \left\langle Bx,x \right\rangle  \right|}^{2}}\le \sqrt{\sqrt{{{\omega }^{4}}\left( A \right)+{{\omega }^{4}}\left( B \right)}\;\omega \left( {{\left| A \right|}^{2}}+{\rm i}{{\left| B \right|}^{2}} \right)+2\omega \left( A \right)\omega \left( B \right)\omega \left( {{B}^{*}}A \right)}.
\end{equation}
Now, we reach the desired result by taking supremum over all unit vectors $x\in \mathbb H$.
\end{proof}
\begin{remark}
In both Theorems \ref{28} and \ref{7}, we have found some upper bounds for $\omega_e(A, B).$ In this remark, we give examples showing that neither bound is uniformly better than the other. For this, let $A=\left[
\begin{array}{cc}
 -3.& 3 \\
 1 & -1 \\
\end{array}
\right]$ and $B=\left[
\begin{array}{cc}
 4 & -5\\
 3 & -5\\
\end{array}
\right].$ Then numerical calculations show that
\[ \sqrt{\sqrt{\left( {{\omega }^{4}}\left( A \right)+{{\omega }^{4}}\left( B \right) \right)}\;\omega \left( {{\left| A \right|}^{2}}+{\rm i}{{\left| B \right|}^{2}} \right)+2\omega \left( A \right)\omega \left( B \right)\omega \left( {{B}^{*}}A \right)}\approx 57.1627,\]
and 
\[\sqrt{\left\| {{\omega }^{2}}\left( A \right){{\left| {{A}^{*}} \right|}^{2}}+{{\omega }^{2}}\left( B \right){{\left| {{B}^{*}} \right|}^{2}} \right\|+2\omega \left( A \right)\omega \left( B \right)\omega \left( A{{B}^{*}} \right)}\approx 57.3063,\]
showing that the bound in Theorem \ref{7} is better than that in Theorem \ref{28} for this example.

On the other hand, if we let $A=\left[
\begin{array}{cc}
 0 & -4\\
 1 & 2 \\
\end{array}
\right]$ and $B=\left[
\begin{array}{cc}
 -3 & 3 \\
 2 & 4 \\
\end{array}
\right]$ we find that
\[ \sqrt{\sqrt{\left( {{\omega }^{4}}\left( A \right)+{{\omega }^{4}}\left( B \right) \right)}\;\omega \left( {{\left| A \right|}^{2}}+{\rm i}{{\left| B \right|}^{2}} \right)+2\omega \left( A \right)\omega \left( B \right)\omega \left( {{B}^{*}}A \right)}\approx 33.1982,\]
and 
\[\sqrt{\left\| {{\omega }^{2}}\left( A \right){{\left| {{A}^{*}} \right|}^{2}}+{{\omega }^{2}}\left( B \right){{\left| {{B}^{*}} \right|}^{2}} \right\|+2\omega \left( A \right)\omega \left( B \right)\omega \left( A{{B}^{*}} \right)}\approx 31.3455.\]
Consequently, Theorems \ref{28} and \ref{7} are generally not comparable.
\end{remark}


An easier bound than that in Theorem \ref{7} can be stated in the following form.

\begin{theorem}\label{thm_main_weee}
Let $A,B\in \mathbb B\left( \mathbb H \right)$. Then
\[\omega _{e}^{2}\left( A,B \right)\le \sqrt{\omega \left( {{\left| A \right|}^{2}}+{\rm i}{{\left| B \right|}^{2}} \right)\omega \left( {{\left| {{A}^{*}} \right|}^{2}}+{\rm i}{{\left| {{B}^{*}} \right|}^{2}} \right)+2\omega \left( A \right)\omega \left( B \right)\omega \left( {{B}^{*}}A \right)}.\]
\end{theorem}
\begin{proof}
By substituting $x=Ax$, $y=Bx$, and $e=x$,  in \eqref{4}, we obtain
{\small
\[\begin{aligned}
  & {{\left| \left\langle Ax,x \right\rangle  \right|}^{2}}+{{\left| \left\langle Bx,x \right\rangle  \right|}^{2}} \\ 
 & \le \sqrt{{{\left| \left\langle Ax,x \right\rangle  \right|}^{2}}{{\left\| Ax \right\|}^{2}}+{{\left| \left\langle Bx,x \right\rangle  \right|}^{2}}{{\left\| Bx \right\|}^{2}}+2\left| \left\langle Ax,x \right\rangle  \right|\left| \left\langle Bx,x \right\rangle  \right|\left| \left\langle {{B}^{*}}Ax,x \right\rangle  \right|} \\ 
 & =\sqrt{{{\left| \left\langle x,{{A}^{*}}x \right\rangle  \right|}^{2}}{{\left\| Ax \right\|}^{2}}+{{\left| \left\langle x,{{B}^{*}}x \right\rangle  \right|}^{2}}{{\left\| Bx \right\|}^{2}}+2\left| \left\langle Ax,x \right\rangle  \right|\left| \left\langle Bx,x \right\rangle  \right|\left| \left\langle {{B}^{*}}Ax,x \right\rangle  \right|} \\ 
 & \le \sqrt{{{\left\| {{A}^{*}}x \right\|}^{2}}{{\left\| Ax \right\|}^{2}}+{{\left\| {{B}^{*}}x \right\|}^{2}}{{\left\| Bx \right\|}^{2}}+2\left| \left\langle Ax,x \right\rangle  \right|\left| \left\langle Bx,x \right\rangle  \right|\left| \left\langle {{B}^{*}}Ax,x \right\rangle  \right|} \\ 
 &\quad \text{(by the Cauchy-Schwarz inequality)}\\
 & \le \sqrt{\sqrt{\left( {{\left\langle {{\left| A \right|}^{2}}x,x \right\rangle }^{2}}+{{\left\langle {{\left| B \right|}^{2}}x,x \right\rangle }^{2}} \right)\left( {{\left\langle {{\left| {{A}^{*}} \right|}^{2}}x,x \right\rangle }^{2}}+{{\left\langle {{\left| {{B}^{*}} \right|}^{2}}x,x \right\rangle }^{2}} \right)}+2\left| \left\langle Ax,x \right\rangle  \right|\left| \left\langle Bx,x \right\rangle  \right|\left| \left\langle {{B}^{*}}Ax,x \right\rangle  \right|} \\ 
  &\quad \text{(by the Cauchy-Schwarz inequality)}\\
 & =\sqrt{\left| \left\langle \left( {{\left| A \right|}^{2}}+{\rm i}{{\left| B \right|}^{2}} \right)x,x \right\rangle  \right|\left| \left\langle \left( {{\left| {{A}^{*}} \right|}^{2}}+{\rm i}{{\left| {{B}^{*}} \right|}^{2}} \right)x,x \right\rangle  \right|+2\left| \left\langle Ax,x \right\rangle  \right|\left| \left\langle Bx,x \right\rangle  \right|\left| \left\langle {{B}^{*}}Ax,x \right\rangle  \right|} \\ 
 & \le \sqrt{\omega \left( {{\left| A \right|}^{2}}+{\rm i}{{\left| B \right|}^{2}} \right)\omega \left( {{\left| {{A}^{*}} \right|}^{2}}+{\rm i}{{\left| {{B}^{*}} \right|}^{2}} \right)+2\omega \left( A \right)\omega \left( B \right)\omega \left( {{B}^{*}}A \right)},
\end{aligned}\]
}
i.e.,
\[{{\left| \left\langle Ax,x \right\rangle  \right|}^{2}}+{{\left| \left\langle Bx,x \right\rangle  \right|}^{2}}\le \sqrt{\omega \left( {{\left| A \right|}^{2}}+{\rm i}{{\left| B \right|}^{2}} \right)\omega \left( {{\left| {{A}^{*}} \right|}^{2}}+{\rm i}{{\left| {{B}^{*}} \right|}^{2}} \right)+2\omega \left( A \right)\omega \left( B \right)\omega \left( {{B}^{*}}A \right)}.\]
Now, we receive the desired result by taking supremum over all unit vectors $x\in \mathbb H$.
\end{proof}
\begin{remark}
If $A=\left[
\begin{array}{cc}
 4 & -2 \\
 -4 & -5 \\
\end{array}
\right]$ and $B=\left[
\begin{array}{cc}
 2 & 5 \\
 2 & 4 \\
\end{array}
\right]$, we see that
\[\sqrt{\omega \left( {{\left| A \right|}^{2}}+{\rm i}{{\left| B \right|}^{2}} \right)\omega \left( {{\left| {{A}^{*}} \right|}^{2}}+{\rm i}{{\left| {{B}^{*}} \right|}^{2}} \right)+2\omega \left( A \right)\omega \left( B \right)\omega \left( {{B}^{*}}A \right)}\approx 76.375,\]
\[ \sqrt{\sqrt{\left( {{\omega }^{4}}\left( A \right)+{{\omega }^{4}}\left( B \right) \right)}\;\omega \left( {{\left| A \right|}^{2}}+{\rm i}{{\left| B \right|}^{2}} \right)+2\omega \left( A \right)\omega \left( B \right)\omega \left( {{B}^{*}}A \right)}\approx 77.146,\]
and 
\[\sqrt{\left\| {{\omega }^{2}}\left( A \right){{\left| {{A}^{*}} \right|}^{2}}+{{\omega }^{2}}\left( B \right){{\left| {{B}^{*}} \right|}^{2}} \right\|+2\omega \left( A \right)\omega \left( B \right)\omega \left( A{{B}^{*}} \right)}\approx 76.3889.\]

However, if we let $A=\left[
\begin{array}{cc}
 -5 & 1 \\
 5 & 3 \\
\end{array}
\right]$ and $B=\left[
\begin{array}{cc}
 -5 & -2 \\
 1 & -4 \\
\end{array}
\right],$ we see that
\[\sqrt{\omega \left( {{\left| A \right|}^{2}}+{\rm i}{{\left| B \right|}^{2}} \right)\omega \left( {{\left| {{A}^{*}} \right|}^{2}}+{\rm i}{{\left| {{B}^{*}} \right|}^{2}} \right)+2\omega \left( A \right)\omega \left( B \right)\omega \left( {{B}^{*}}A \right)}\approx 72.465,\]
\[ \sqrt{\sqrt{\left( {{\omega }^{4}}\left( A \right)+{{\omega }^{4}}\left( B \right) \right)}\;\omega \left( {{\left| A \right|}^{2}}+{\rm i}{{\left| B \right|}^{2}} \right)+2\omega \left( A \right)\omega \left( B \right)\omega \left( {{B}^{*}}A \right)}\approx 67.9146,\]
and 
\[\sqrt{\left\| {{\omega }^{2}}\left( A \right){{\left| {{A}^{*}} \right|}^{2}}+{{\omega }^{2}}\left( B \right){{\left| {{B}^{*}} \right|}^{2}} \right\|+2\omega \left( A \right)\omega \left( B \right)\omega \left( A{{B}^{*}} \right)}\approx 66.9069.\]
These two examples show that  Theorem \ref{thm_main_weee} is not comparable, in general, with Theorems \ref{28} and \ref{7}. 
\end{remark}

\subsection{Upper bounds for $\|(\cdot,\cdot)\|_e$}
Now we discuss some possible bounds for $\|(\cdot,\cdot)\|_e$.
\begin{theorem}\label{thm_main_opere}
Let $A,B\in \mathbb B\left( \mathbb H \right)$. Then
\[\left\| \left( A,B \right) \right\|_{e}^{2}\le \sqrt{\left\| {{\left\| A \right\|}^{2}}{{\left| A \right|}^{2}}+{{\left\| B \right\|}^{2}}{{\left| B \right|}^{2}} \right\|+\frac{1}{2}\left\| \;\left| A \right|+\left| B \right| \;\right\|\left\| \;\left| {{A}^{*}} \right|+\left| {{B}^{*}} \right| \;\right\|\omega \left( {{A}^{*}}B \right)}.\]
\end{theorem}
\begin{proof}
Letting $y={{A}^{*}}y$, $z={{B}^{*}}y$, and $\left\| x \right\|=1$, in \eqref{26}, we infer that
\[\begin{aligned}
  & {{\left| \left\langle Ax,y \right\rangle  \right|}^{2}}+{{\left| \left\langle Bx,y \right\rangle  \right|}^{2}} \\ 
 & \le \left\| \left( \left\langle Ax,y \right\rangle {{A}^{*}}+\left\langle Bx,y \right\rangle {{B}^{*}} \right)y \right\| \\ 
 & \le \left\| \left\langle Ax,y \right\rangle {{A}^{*}}+\left\langle Bx,y \right\rangle {{B}^{*}} \right\| \\ 
 & ={{\left\| \left( \left\langle Ax,y \right\rangle {{A}^{*}}+\left\langle Bx,y \right\rangle {{B}^{*}} \right)\left( \overline{\left\langle Ax,y \right\rangle }A+\overline{\left\langle Bx,y \right\rangle }B \right) \right\|}^{\frac{1}{2}}} \\ 
  &\quad \text{(since $\left\| X{{X}^{*}} \right\|={{\left\| X \right\|}^{2}}$ for any $X\in \mathbb B\left( \mathbb H \right)$)}\\
 & ={{\left\| {{\left| \left\langle Ax,y \right\rangle  \right|}^{2}}{{\left| A \right|}^{2}}+{{\left| \left\langle Bx,y \right\rangle  \right|}^{2}}{{\left| B \right|}^{2}}+2\Re\left( \left\langle Ax,y \right\rangle \overline{\left\langle Bx,y \right\rangle }{{A}^{*}}B \right) \right\|}^{\frac{1}{2}}} \\ 
 & \le \sqrt{\left\| {{\left| \left\langle Ax,y \right\rangle  \right|}^{2}}{{\left| A \right|}^{2}}+{{\left| \left\langle Bx,y \right\rangle  \right|}^{2}}{{\left| B \right|}^{2}} \right\|+2\left\| \Re\left( \left\langle Ax,y \right\rangle \overline{\left\langle Bx,y \right\rangle }{{A}^{*}}B \right) \right\|} \\ 
  &\quad \text{(by the triangle inequality for the usual operator norm)}\\
 & \le \sqrt{\left\| {{\left| \left\langle Ax,y \right\rangle  \right|}^{2}}{{\left| A \right|}^{2}}+{{\left| \left\langle Bx,y \right\rangle  \right|}^{2}}{{\left| B \right|}^{2}} \right\|+2\left| \left\langle Ax,y \right\rangle \overline{\left\langle Bx,y \right\rangle } \right|\omega \left( {{A}^{*}}B \right)}. 
\end{aligned}\]
On the other hand,
\[\begin{aligned}
  & \left\| {{\left| \left\langle Ax,y \right\rangle  \right|}^{2}}{{\left| A \right|}^{2}}+{{\left| \left\langle Bx,y \right\rangle  \right|}^{2}}{{\left| B \right|}^{2}} \right\|+2\left| \left\langle Ax,y \right\rangle  \right|\left| \left\langle Bx,y \right\rangle  \right|\omega \left( {{A}^{*}}B \right) \\ 
 & \le \left\| {{\left\| A \right\|}^{2}}{{\left| A \right|}^{2}}+{{\left\| B \right\|}^{2}}{{\left| B \right|}^{2}} \right\|+2\left| \left\langle Ax,y \right\rangle  \right|\left| \left\langle Bx,y \right\rangle  \right|\omega \left( {{A}^{*}}B \right) \\ 
 & \le \left\| {{\left\| A \right\|}^{2}}{{\left| A \right|}^{2}}+{{\left\| B \right\|}^{2}}{{\left| B \right|}^{2}} \right\|+2\sqrt{\left\langle \left| A \right|x,x \right\rangle \left\langle \left| {{A}^{*}} \right|y,y \right\rangle }\sqrt{\left\langle \left| B \right|x,x \right\rangle \left\langle \left| {{B}^{*}} \right|y,y \right\rangle }\omega \left( {{A}^{*}}B \right) \\ 
 &\quad \text{(by the mixed Schwarz inequality)}\\
 & \le \left\| {{\left\| A \right\|}^{2}}{{\left| A \right|}^{2}}+{{\left\| B \right\|}^{2}}{{\left| B \right|}^{2}} \right\|+\frac{1}{2}\left\langle \left( \left| A \right|+\left| B \right| \right)x,x \right\rangle \left\langle \left( \left| {{A}^{*}} \right|+\left| {{B}^{*}} \right| \right)y,y \right\rangle \omega \left( {{A}^{*}}B \right) \\ 
  &\quad \text{(by the arithmetic-geometric mean inequality)}\\
 & \le \left\| {{\left\| A \right\|}^{2}}{{\left| A \right|}^{2}}+{{\left\| B \right\|}^{2}}{{\left| B \right|}^{2}} \right\|+\frac{1}{2}\left\| \;\left| A \right|+\left| B \right| \;\right\|\left\| \;\left| {{A}^{*}} \right|+\left| {{B}^{*}} \right| \;\right\|\omega \left( {{A}^{*}}B \right).
\end{aligned}\]
Accordingly,
\[{{\left| \left\langle Ax,y \right\rangle  \right|}^{2}}+{{\left| \left\langle Bx,y \right\rangle  \right|}^{2}}\le \sqrt{\left\| {{\left\| A \right\|}^{2}}{{\left| A \right|}^{2}}+{{\left\| B \right\|}^{2}}{{\left| B \right|}^{2}} \right\|+\frac{1}{2}\left\| \;\left| A \right|+\left| B \right| \;\right\|\left\| \;\left| {{A}^{*}} \right|+\left| {{B}^{*}} \right| \;\right\|\omega \left( {{A}^{*}}B \right)},\]
which implies the desired inequality after taking supremum over all unit vectors $x,y\in \mathbb H$.
\end{proof}

\begin{theorem}\label{thm_main_operee}
Let $A,B\in \mathbb B\left( \mathbb H \right)$. Then
\[\left\| \left( A,B \right) \right\|_{e}^{2}\le \sqrt{\omega \left( {{\left| A \right|}^{2}}+{\rm i}{{\left| B \right|}^{2}} \right)\omega \left( {{\left| {{A}^{*}} \right|}^{2}}+{\rm i}{{\left| {{B}^{*}} \right|}^{2}} \right)+\omega \left( \left| A \right|+{\rm i}\left| B \right| \right)\omega \left( \left| {{A}^{*}} \right|+{\rm i}\left| {{B}^{*}} \right| \right)\omega \left( B{{A}^{*}} \right)}.\]
\end{theorem}
\begin{proof}
Letting $y={{A}^{*}}y$ and $z={{B}^{*}}y$, with $\left\| x \right\|=\left\| y \right\|=1$, in \eqref{5}, we obtain
{\tiny	
\[\begin{aligned}
  & {{\left| \left\langle Ax,y \right\rangle  \right|}^{2}}+{{\left| \left\langle Bx,y \right\rangle  \right|}^{2}} \\ 
 & \le \sqrt{{{\left| \left\langle Ax,y \right\rangle  \right|}^{2}}{{\left\| {{A}^{*}}y \right\|}^{2}}+{{\left| \left\langle Bx,y \right\rangle  \right|}^{2}}{{\left\| {{B}^{*}}y \right\|}^{2}}+2\left| \left\langle Ax,y \right\rangle  \right|\left| \left\langle Bx,y \right\rangle  \right|\left| \left\langle B{{A}^{*}}y,y \right\rangle  \right|} \\ 
 & \le \sqrt{{{\left\| Ax \right\|}^{2}}{{\left\| {{A}^{*}}y \right\|}^{2}}+{{\left\| Bx \right\|}^{2}}{{\left\| {{B}^{*}}y \right\|}^{2}}+2\left| \left\langle Ax,y \right\rangle  \right|\left| \left\langle Bx,y \right\rangle  \right|\left| \left\langle B{{A}^{*}}y,y \right\rangle  \right|} \\ 
 &\quad \text{(by the Cauchy-Schwarz inequality)}\\
 & =\sqrt{\left\langle {{\left| A \right|}^{2}}x,x \right\rangle \left\langle {{\left| {{A}^{*}} \right|}^{2}}y,y \right\rangle +\left\langle {{\left| B \right|}^{2}}x,x \right\rangle \left\langle {{\left| {{B}^{*}} \right|}^{2}}y,y \right\rangle +2\left| \left\langle Ax,y \right\rangle  \right|\left| \left\langle Bx,y \right\rangle  \right|\left| \left\langle B{{A}^{*}}y,y \right\rangle  \right|} \\ 
 & \le \sqrt{\left\langle {{\left| A \right|}^{2}}x,x \right\rangle \left\langle {{\left| {{A}^{*}} \right|}^{2}}y,y \right\rangle +\left\langle {{\left| B \right|}^{2}}x,x \right\rangle \left\langle {{\left| {{B}^{*}} \right|}^{2}}y,y \right\rangle +2\sqrt{\left\langle \left| A \right|x,x \right\rangle \left\langle \left| {{A}^{*}} \right|y,y \right\rangle \left\langle \left| B \right|x,x \right\rangle \left\langle \left| {{B}^{*}} \right|y,y \right\rangle }\left| \left\langle B{{A}^{*}}y,y \right\rangle  \right|} \\ 
  &\quad \text{(by the mixed Schwarz inequality)}\\
 & \le \sqrt{\left\langle {{\left| A \right|}^{2}}x,x \right\rangle \left\langle {{\left| {{A}^{*}} \right|}^{2}}y,y \right\rangle +\left\langle {{\left| B \right|}^{2}}x,x \right\rangle \left\langle {{\left| {{B}^{*}} \right|}^{2}}y,y \right\rangle +\left( \left\langle \left| A \right|x,x \right\rangle \left\langle \left| {{A}^{*}} \right|y,y \right\rangle +\left\langle \left| B \right|x,x \right\rangle \left\langle \left| {{B}^{*}} \right|y,y \right\rangle  \right)\left| \left\langle B{{A}^{*}}y,y \right\rangle  \right|} \\ 
   &\quad \text{(by the arithmetic-geometric mean inequality)}\\
 & \le \sqrt{\left( {{\left\langle {{\left| A \right|}^{2}}x,x \right\rangle }^{2}}+{{\left\langle {{\left| B \right|}^{2}}x,x \right\rangle }^{2}} \right)^{\frac{1}{2}}\left( {{\left\langle {{\left| {{A}^{*}} \right|}^{2}}y,y \right\rangle }^{2}}+{{\left\langle {{\left| {{B}^{*}} \right|}^{2}}y,y \right\rangle }^{2}} \right)^{\frac{1}{2}}+\left( {{\left\langle \left| A \right|x,x \right\rangle }^{2}}+{{\left\langle \left| B \right|x,x \right\rangle }^{2}} \right)^{\frac{1}{2}}\left( {{\left\langle \left| {{A}^{*}} \right|y,y \right\rangle }^{2}}+{{\left\langle \left| {{B}^{*}} \right|y,y \right\rangle }^{2}} \right)^{\frac{1}{2}}\left| \left\langle B{{A}^{*}}y,y \right\rangle  \right|} \\ 
  &\quad \text{(by the Cauchy-Schwarz inequality)}\\
 & =\sqrt{\left| \left\langle \left( {{\left| A \right|}^{2}}+{\rm i}{{\left| B \right|}^{2}} \right)x,x \right\rangle  \right|\left| \left\langle \left( {{\left| {{A}^{*}} \right|}^{2}}+{\rm i}{{\left| {{B}^{*}} \right|}^{2}} \right)y,y \right\rangle  \right|+\left| \left\langle \left( \left| A \right|+{\rm i}\left| B \right| \right)x,x \right\rangle  \right|\left| \left\langle \left( \left| {{A}^{*}} \right|+{\rm i}\left| {{B}^{*}} \right| \right)y,y \right\rangle  \right|\left| \left\langle B{{A}^{*}}y,y \right\rangle  \right|} \\ 
   &\quad \text{(since $\left| a+{\rm i}b \right|=\sqrt{{{a}^{2}}+{{b}^{2}}}$ for any $a,b\in \mathbb{R}$)}\\
 & \le\sqrt{\omega \left( {{\left| A \right|}^{2}}+{\rm i}{{\left| B \right|}^{2}} \right)\omega \left( {{\left| {{A}^{*}} \right|}^{2}}+{\rm i}{{\left| {{B}^{*}} \right|}^{2}} \right)+\omega \left( \left| A \right|+{\rm i}\left| B \right| \right)\omega \left( \left| {{A}^{*}} \right|+{\rm i}\left| {{B}^{*}} \right| \right)\omega \left( B{{A}^{*}} \right)},  
\end{aligned}\]
}
i.e.,
\begin{equation}\label{14}
\begin{aligned}
  & {{\left| \left\langle Ax,y \right\rangle  \right|}^{2}}+{{\left| \left\langle Bx,y \right\rangle  \right|}^{2}} \\ 
 & \le \sqrt{\omega \left( {{\left| A \right|}^{2}}+{\rm i}{{\left| B \right|}^{2}} \right)\omega \left( {{\left| {{A}^{*}} \right|}^{2}}+{\rm i}{{\left| {{B}^{*}} \right|}^{2}} \right)+\omega \left( \left| A \right|+{\rm i}\left| B \right| \right)\omega \left( \left| {{A}^{*}} \right|+{\rm i}\left| {{B}^{*}} \right| \right)\omega \left( B{{A}^{*}} \right)}.  
\end{aligned}
\end{equation}
Now, we receive the desired result by taking supremum over all unit vectors $x,y\in \mathbb H$.
\end{proof}
\begin{remark}
In this remark, we compare the two bounds found in Theorems \ref{thm_main_opere} and \ref{thm_main_operee}. Indeed, letting $A=\left[
\begin{array}{cc}
 0 & 1 \\
 1 & 0 \\
\end{array}
\right]$ and $B=\left[
\begin{array}{cc}
 1 & 1 \\
 0 & 1 \\
\end{array}
\right]$, we find that
\[\sqrt{\left\| {{\left\| A \right\|}^{2}}{{\left| A \right|}^{2}}+{{\left\| B \right\|}^{2}}{{\left| B \right|}^{2}} \right\|+\frac{1}{2}\left\| \;\left| A \right|+\left| B \right| \;\right\|\left\| \;\left| {{A}^{*}} \right|+\left| {{B}^{*}} \right| \;\right\|\omega \left( {{A}^{*}}B \right)}\approx 3.02706,\]
and
\[\sqrt{\omega \left( {{\left| A \right|}^{2}}+{\rm i}{{\left| B \right|}^{2}} \right)\omega \left( {{\left| {{A}^{*}} \right|}^{2}}+{\rm i}{{\left| {{B}^{*}} \right|}^{2}} \right)+\omega \left( \left| A \right|+{\rm i}\left| B \right| \right)\omega \left( \left| {{A}^{*}} \right|+{\rm i}\left| {{B}^{*}} \right| \right)\omega \left( B{{A}^{*}} \right)}\approx 3.70246.\]

On the other hand, letting $A=\left[
\begin{array}{cc}
 1 & 1 \\
 0 & 0 \\
\end{array}
\right],$ and $B=\left[
\begin{array}{cc}
 0 & 1 \\
 1 & 0 \\
\end{array}
\right],$ we find
\[\sqrt{\left\| {{\left\| A \right\|}^{2}}{{\left| A \right|}^{2}}+{{\left\| B \right\|}^{2}}{{\left| B \right|}^{2}} \right\|+\frac{1}{2}\left\| \;\left| A \right|+\left| B \right| \;\right\|\left\| \;\left| {{A}^{*}} \right|+\left| {{B}^{*}} \right| \;\right\|\omega \left( {{A}^{*}}B \right)}\approx 3.08509,\]
and
\[\sqrt{\omega \left( {{\left| A \right|}^{2}}+{\rm i}{{\left| B \right|}^{2}} \right)\omega \left( {{\left| {{A}^{*}} \right|}^{2}}+{\rm i}{{\left| {{B}^{*}} \right|}^{2}} \right)+\omega \left( \left| A \right|+{\rm i}\left| B \right| \right)\omega \left( \left| {{A}^{*}} \right|+{\rm i}\left| {{B}^{*}} \right| \right)\omega \left( B{{A}^{*}} \right)}\approx 2.93621.\]

These examples show that neither Theorem \ref{thm_main_opere} nor Theorem \ref{thm_main_operee} is uniformly better than the other.
\end{remark}


\subsection{Applications towards the numerical radius}\label{section_w}
If we put $A=T$ and $B={{T}^{*}}$, in Theorem \ref{2}, we reach to the following result due to Dragomir \cite[Theorem 1]{3}.
\begin{corollary}\label{8}
Let $T\in \mathbb B\left( \mathbb H \right)$. Then
\begin{equation}\label{eq_upper_bound_2}
{{\omega }^{2}}\left( T \right)\le \frac{1}{2}{{\left\| T \right\|}^{2}}+\frac{1}{2}\omega \left( {{T}^{2}} \right).
\end{equation}

\end{corollary}
\begin{remark}
We know from \eqref{eq_kitt_2} that
\begin{equation}\label{eq_kitt_square}
\omega^2(T)\leq \frac{1}{2}\left\| |T|^2+|T^*|^2\right\|.
\end{equation}
In this remark, we give two examples to show that neither this bound nor the bound found in Corollary \ref{8} is uniformly better than the other. 
\begin{itemize}
\item[(i)] 

If we take $T=\left[
\begin{array}{cc}
 -5 & 1 \\
 4 & 4 \\
\end{array}
\right],$ then direct calculations illustrate that
\[\frac{1}{2}\left\| |T|^2+|T^*|^2\right\|\approx 34.1478\;{\text{and}}\;\frac{1}{2}{{\left\| T \right\|}^{2}}+\frac{1}{2}\omega \left( {{T}^{2}} \right)\approx 37.4633.\] This shows, for this $T$, that \eqref{eq_kitt_square} is better than \eqref{eq_upper_bound_2}.

However, if we take $T=\left[
\begin{array}{cc}
 1 & 0 \\
 1 & 0 \\
\end{array}
\right],$ we find that
\[\frac{1}{2}\left\| |T|^2+|T^*|^2\right\|\approx 1.70711\;{\text{and}}\;\frac{1}{2}{{\left\| T \right\|}^{2}}+\frac{1}{2}\omega \left( {{T}^{2}} \right)\approx 1.60355,\]
showing that \eqref{eq_upper_bound_2} is sharper than \eqref{eq_kitt_square}.

\item[(ii)] Now we provide an example to show that \eqref{eq_upper_bound_2} can be sharper than \eqref{eq_kitt_1}. Taking $T=\left[
\begin{array}{cc}
 2 & 0 \\
 1 & 5 \\
\end{array}
\right]$ yields
\[\frac{1}{2}{{\left\| T \right\|}^{2}}+\frac{1}{2}\omega \left( {{T}^{2}} \right)\approx 25.8742\;{\text{and}}\;\left(\frac{1}{2}\left\| \;|T|+|T^*|\;\right\|\right)^2\approx 26.018.\]
However, letting $T=\left[
\begin{array}{cc}
 3 & 0 \\
 4 & 1 \\
\end{array}
\right]$ gives
\[\frac{1}{2}{{\left\| T \right\|}^{2}}+\frac{1}{2}\omega \left( {{T}^{2}} \right)\approx 19.7967\;{\text{and}}\;\left(\frac{1}{2}\left\| \;|T|+|T^*|\;\right\|\right)^2\approx 19.4443.\]
\end{itemize}

\end{remark}
If we set $A=T$ and $B={{T}^{*}}$, in Theorem \ref{7}, we obtain the following bound for $\omega(T).$
\begin{corollary}\label{9}
Let $T\in \mathbb B\left( \mathbb H \right)$. Then
\begin{equation}\label{18}
\omega \left( T \right)\le \frac{1}{2}\sqrt{\sqrt{2}\omega \left( {{\left| T \right|}^{2}}+{\rm i}{{\left| {{T}^{*}} \right|}^{2}} \right)+2\omega \left( {{T}^{2}} \right)}.
\end{equation}
\end{corollary}

\begin{remark}
\hfill
\begin{itemize}
\item[(i)] Notice that \eqref{18} is sharp. Indeed, if we assume that $T$ is a normal operator, we get the same quantity $\left\| T \right\|$ on both sides of the inequality.
\item[(ii)] Of course, the inequality in Corollary \ref{9} is stronger than the inequality in Corollary \ref{8}, since we have (see, e.g., \cite[Proposition 1.4]{4})
\[\omega \left( {{\left| T \right|}^{2}}+{\rm i}{{\left| {{T}^{*}} \right|}^{2}} \right)\le {{\left\| {{\left| T \right|}^{4}}+{{\left| {{T}^{*}} \right|}^{4}} \right\|}^{\frac{1}{2}}}.\]

\item[(iii)] We explain the advantage of the bound in \eqref{18}. If we let
$T=\left[
\begin{array}{cc}
 1 & 0 \\
 4 & 1 \\
\end{array}
\right],$ we can see that $\omega(T)=3$. Moreover,
\[ \frac{1}{2}\left\| \;|T|+|T^*|\;\right\|\approx 3.1305,\;{\text{and}}\;
\frac{1}{2}\sqrt{\sqrt{2}\omega \left( {{\left| T \right|}^{2}}+{\rm i}{{\left| {{T}^{*}} \right|}^{2}} \right)+2\omega \left( {{T}^{2}} \right)}\approx 3.00956,\]
which shows that \eqref{18} is sharper than \eqref{eq_kitt_1}, in this example. However, if we let $T=\left[
\begin{array}{cc}
 0 & 3 \\
 0 & 0 \\
\end{array}
\right]$, then \eqref{eq_kitt_1} is sharper than \eqref{18}, as we have $\omega(T)\approx 1.5,$
\[ \frac{1}{2}\left\| \;|T|+|T^*|\;\right\|= 1.5,\;{\text{and}}\;
\frac{1}{2}\sqrt{\sqrt{2}\omega \left( {{\left| T \right|}^{2}}+{\rm i}{{\left| {{T}^{*}} \right|}^{2}} \right)+2\omega \left( {{T}^{2}} \right)}\approx 1.78381.\]
\end{itemize}
\end{remark}

Letting $T=\left[ \begin{matrix}
   O & A  \\
   {{B}^{*}} & O  \\
\end{matrix} \right]$, in Corollary \ref{9}, obtain the following upper bound for the numerical radius of the operator matrix $ \left[ \begin{matrix}
   O & A  \\
   {{B}^{*}} & O  \\
\end{matrix} \right] .$
\begin{corollary}\label{19}
Let $A,B\in \mathbb B\left( \mathbb H \right)$. Then
\[\begin{aligned}
  & {{\omega }^{2}}\left( \left[ \begin{matrix}
   O & A  \\
   {{B}^{*}} & O  \\
\end{matrix} \right] \right) \\ 
 & \le \frac{\sqrt{2}}{4}\max \left\{ \omega \left( {{\left| {{A}^{*}} \right|}^{2}}+{\rm i}{{\left| {{B}^{*}} \right|}^{2}} \right),\omega \left( {{\left| A \right|}^{2}}+{\rm i}{{\left| B \right|}^{2}} \right) \right\}+\frac{1}{2}\max \left\{ \omega \left( A{{B}^{*}} \right),\omega \left( {{B}^{*}}A \right) \right\}.
 \end{aligned}\]
\end{corollary}

\begin{remark}\label{rem_main_block}
\hfill
\begin{itemize}
\item[(i)] The inequality in Corollary \ref{19} is sharp. Indeed, if we let $B^{*}=A$ be a normal operator, we get ${{\left\| A \right\|}^{2}}$ on both sides of the inequality.

\item[(ii)] Among the sharpest upper bounds for $\omega\left( \left[\begin{matrix}O&A\\B^*&O\end{matrix}\right]     \right)$ is $\frac{\|A\|+\|B\|}{2},$ as we see from Lemma \ref{lem_hirz}. In Corollary \ref{19}, we have found a new upper bound. We give two examples to show that neither bound is uniformly better. For, let 
\[A=\left[
\begin{array}{cc}
 5 & 0 \\
 2 & 5 \\
\end{array}
\right]\;{\text{and}}\;B=\left[
\begin{array}{cc}
 1 & 0 \\
 1 & 3\\
\end{array}
\right].\]
Then
\[\omega^2\left( \left[\begin{matrix}O&A\\B^*&O\end{matrix}\right]     \right)\approx 20.078,\left(\frac{\|A\|+\|B\|}{2}\right)^2\approx 21.5231,\]
and
\[\frac{\sqrt{2}}{4}\max \left\{ \omega \left( {{\left| {{A}^{*}} \right|}^{2}}+{\rm i}{{\left| {{B}^{*}} \right|}^{2}} \right),\omega \left( {{\left| A \right|}^{2}}+{\rm i}{{\left| B \right|}^{2}} \right) \right\}+\frac{1}{2}\max \left\{ \omega \left( A{{B}^{*}} \right),\omega \left( {{B}^{*}}A \right) \right\}\approx 22.4192.\]
On the other hand, if \[A=\left[
\begin{array}{cc}
 5 & 0 \\
 1 & 2 \\
\end{array}
\right]\;{\text{and}}\;B=\left[
\begin{array}{cc}
 5 & 4 \\
 4 & 0 \\
\end{array}
\right],\]
then
\[\omega^2\left( \left[\begin{matrix}O&A\\B^*&O\end{matrix}\right]     \right)\approx 36.25,\left(\frac{\|A\|+\|B\|}{2}\right)^2\approx 38.0298,\]
and
\[\frac{\sqrt{2}}{4}\max \left\{ \omega \left( {{\left| {{A}^{*}} \right|}^{2}}+{\rm i}{{\left| {{B}^{*}} \right|}^{2}} \right),\omega \left( {{\left| A \right|}^{2}}+{\rm i}{{\left| B \right|}^{2}} \right) \right\}+\frac{1}{2}\max \left\{ \omega \left( A{{B}^{*}} \right),\omega \left( {{B}^{*}}A \right) \right\}\approx 37.7279.\]
\end{itemize}
\end{remark}

\begin{proposition}\label{13}
Let $T\in \mathbb B\left( \mathbb H \right)$. Then
\[{{\left\| \Re T \right\|}^{2}}\le \frac{1}{4}\sqrt{\sqrt{2}{{\omega }^{2}}\left( T \right)\omega \left( {{\left| T \right|}^{2}}+{\rm i}{{\left| {{T}^{*}} \right|}^{2}} \right)+2{{\omega }^{2}}\left( T \right)\omega \left( {{T}^{2}} \right)}+\frac{1}{2}{{\omega }^{2}}\left( T \right).\]
\end{proposition}
\begin{proof}
It observes from the inequality \eqref{12} that
\[\begin{aligned}
  & {{\left| \left\langle \left( A+B \right)x,x \right\rangle  \right|}^{2}} \\ 
 & \le {{\left| \left\langle Ax,x \right\rangle  \right|}^{2}}+{{\left| \left\langle Bx,x \right\rangle  \right|}^{2}}+2\left| \left\langle Ax,x \right\rangle  \right|\left| \left\langle Bx,x \right\rangle  \right| \\ 
 & \le \sqrt{\sqrt{{{\omega }^{4}}\left( A \right)+{{\omega }^{4}}\left( B \right)}\;\omega \left( {{\left| A \right|}^{2}}+{\rm i}{{\left| B \right|}^{2}} \right)+2\omega \left( A \right)\omega \left( B \right)\omega \left( {{B}^{*}}A \right)}+2\left| \left\langle Ax,x \right\rangle  \right|\left| \left\langle Bx,x \right\rangle  \right| \\ 
 & \le \sqrt{\sqrt{{{\omega }^{4}}\left( A \right)+{{\omega }^{4}}\left( B \right)}\;\omega \left( {{\left| A \right|}^{2}}+{\rm i}{{\left| B \right|}^{2}} \right)+2\omega \left( A \right)\omega \left( B \right)\omega \left( {{B}^{*}}A \right)}+2\omega \left( A \right)\omega \left( B \right),
\end{aligned}\]
i.e.,
\[\begin{aligned}
  & {{\left| \left\langle \left( A+B \right)x,x \right\rangle  \right|}^{2}} \\ 
 & \le \sqrt{\sqrt{{{\omega }^{4}}\left( A \right)+{{\omega }^{4}}\left( B \right)}\;\omega \left( {{\left| A \right|}^{2}}+{\rm i}{{\left| B \right|}^{2}} \right)+2\omega \left( A \right)\omega \left( B \right)\omega \left( {{B}^{*}}A \right)}+2\omega \left( A \right)\omega \left( B \right) 
\end{aligned}\]
which implies 
\begin{equation}\label{16}
\begin{aligned}
  & {{\omega }^{2}}\left( A+B \right) \\ 
 & \le \sqrt{\sqrt{{{\omega }^{4}}\left( A \right)+{{\omega }^{4}}\left( B \right)}\;\omega \left( {{\left| A \right|}^{2}}+{\rm i}{{\left| B \right|}^{2}} \right)+2\omega \left( A \right)\omega \left( B \right)\omega \left( {{B}^{*}}A \right)}+2\omega \left( A \right)\omega \left( B \right). 
\end{aligned}
\end{equation}
If we placed $A=T$ and $B={{T}^{*}}$, in \ref{16}, we get the desired result.
\end{proof}
\begin{remark}
Replacing $T$ by $e^{{\textup{i}}\theta}T$ in Proposition \ref{13}, and using the fact that $\omega(T)=\sup_{\theta\in\mathbb{R}}\left\|\Re\left(e^{{\textup{i}}\theta}T\right)\right\|$ (see, e.g., \cite[(2.3)]{Yamazaki}), we obtain Corollary \ref{9}.
\end{remark}

Now, if we replace $A$ by $T$ and $B$ by $T^*$ in Theorem \ref{thm_main_weee}, then we use the fact that $\omega_e(T,T^*)=\sqrt{2}\omega(T)$, we obtain the following, upon noting the equality \[\omega\left(|T|^2+{\textup{i}}|T^*|^2\right)=\omega\left(|T|^2-{\textup{i}}|T^*|^2\right).\]
\begin{corollary}\label{20}
Let $T\in \mathbb B\left( \mathbb H \right)$. Then
\[{{\omega }^{2}}\left( T \right)\le \frac{1}{2}\sqrt{{{\omega }^{2}}\left( {{\left| T \right|}^{2}}+{\rm i}{{\left| {{T}^{*}} \right|}^{2}} \right)+2{{\omega }^{2}}\left( T \right)\omega \left( {{T}^{2}} \right)}.\]
\end{corollary}

\begin{remark}
\hfill
\begin{itemize}
\item[(i)] The inequality in Corollary \ref{20} is sharp. Indeed, if we assume that $T$ is a normal operator, we get the same quantity $\left\| T \right\|$ on both sides of the inequality.

\item[(ii)] In Corollary \ref{20}, we have found an upper bound for $\omega^2(T)$. Here, we provide two examples to show that neither this new bound nor the celebrated bound from \eqref{eq_kitt_1} is always better than the other. Indeed, if $T=\left[
\begin{array}{cc}
 1 & 2 \\
 0 & 4 \\
\end{array}
\right],$ then
\[\omega^2(T)\approx 18.5139, \left(\frac{1}{2}\left\|\; |T|+|T^*|\;\right\|\right)^2\approx 19.0656,\]
and
\[\frac{1}{2}\sqrt{{{\omega }^{2}}\left( {{\left| T \right|}^{2}}+{\rm i}{{\left| {{T}^{*}} \right|}^{2}} \right)+2{{\omega }^{2}}\left( T \right)\omega \left( {{T}^{2}} \right)}\approx 18.7755.\]
This shows that, for this example, the bound found in Corollary \ref{20} is sharper than that in \eqref{eq_kitt_1}. On the other hand, if we let $T=\left[
\begin{array}{cc}
 0 & 3 \\
 4 & 2 \\
\end{array}
\right],$ then
\[\omega^2(T)\approx 21.5301, \left(\frac{1}{2}\left\| \;|T|+|T^*|\;\right\|\right)^2\approx 21.5301,\]
and
\[\frac{1}{2}\sqrt{{{\omega }^{2}}\left( {{\left| T \right|}^{2}}+{\rm i}{{\left| {{T}^{*}} \right|}^{2}} \right)+2{{\omega }^{2}}\left( T \right)\omega \left( {{T}^{2}} \right)}\approx 21.5932.\]

\item[(iii)] In Corollary \ref{8}, we have found another upper bound for $\omega^2(T).$ If we let 
$T=\left[
\begin{array}{cc}
 1 & 1 \\
 5 & 1 \\
\end{array}
\right],$ then it can be verified that
\[\frac{1}{2}{{\left\| T \right\|}^{2}}+\frac{1}{2}\omega \left( {{T}^{2}} \right)\approx  19.7082,\;{\text{and}}\;\frac{1}{2}\sqrt{{{\omega }^{2}}\left( {{\left| T \right|}^{2}}+{\rm i}{{\left| {{T}^{*}} \right|}^{2}} \right)+2{{\omega }^{2}}\left( T \right)\omega \left( {{T}^{2}} \right)}\approx 17.282,\]
showing that the bound we found in Corollary \ref{20} is sharper than that we found in Corollary \ref{8}, for this example. However, letting $T=\left[
\begin{array}{cc}
 2 & 0 \\
 1 & 5 \\
\end{array}
\right]$ implies
\[\frac{1}{2}{{\left\| T \right\|}^{2}}+\frac{1}{2}\omega \left( {{T}^{2}} \right)\approx  25.8742,\;{\text{and}}\;\frac{1}{2}\sqrt{{{\omega }^{2}}\left( {{\left| T \right|}^{2}}+{\rm i}{{\left| {{T}^{*}} \right|}^{2}} \right)+2{{\omega }^{2}}\left( T \right)\omega \left( {{T}^{2}} \right)}\approx 25.881.\]

\item[(iv)] Letting $T=\left[
\begin{array}{cc}
 0 & 3 \\
 0 & 0 \\
\end{array}
\right]$ verifies that
{\footnotesize	
\[\frac{1}{2}\sqrt{{{\omega }^{2}}\left( {{\left| T \right|}^{2}}+{\rm i}{{\left| {{T}^{*}} \right|}^{2}} \right)+2{{\omega }^{2}}\left( T \right)\omega \left( {{T}^{2}} \right)}\approx 4.5\;{\text{and}}\;\left(\frac{1}{2}\sqrt{\sqrt{2}\omega \left( {{\left| T \right|}^{2}}+{\rm i}{{\left| {{T}^{*}} \right|}^{2}} \right)+2\omega \left( {{T}^{2}} \right)}\right)^2\approx 3.18198,\]
}
which shows that the bound we found in Corollary \ref{9} is sharper than the one found in Corollary \ref{20}, in this case.  
\end{itemize}
\end{remark}
\begin{corollary}\label{17}
Let $A,B\in \mathbb B\left( \mathbb H \right)$. Then
{\small
\[\begin{aligned}
   {{\omega }^{2}}\left( \left[ \begin{matrix}
   O & A  \\
   {{B}^{*}} & O  \\
\end{matrix} \right] \right)&\le \frac{1}{4}\sqrt{\omega \left( {{\left| A \right|}^{2}}+{\rm i}{{\left| B \right|}^{2}} \right)\omega \left( {{\left| {{A}^{*}} \right|}^{2}}+{\rm i}{{\left| {{B}^{*}} \right|}^{2}} \right)+\omega \left( \left| A \right|+{\rm i}\left| B \right| \right)\omega \left( \left| {{A}^{*}} \right|+{\rm i}\left| {{B}^{*}} \right| \right)\omega \left( B{{A}^{*}} \right)} \\ 
 &\qquad +\frac{1}{4}\sqrt{\left\| \left| A \right|^2+\left| B \right|^2 \right\|\left\| \left| {{A}^{*}} \right|^2+\left| {{B}^{*}} \right|^2 \right\|}.  
\end{aligned}\]
}
\end{corollary}
\begin{proof}
From \eqref{14}, we have
\[\begin{aligned}
  & {{\left| \left\langle \left( A+B \right)x,y \right\rangle  \right|}^{2}} \\ 
 & \le {{\left| \left\langle Ax,y \right\rangle  \right|}^{2}}+{{\left| \left\langle Bx,y \right\rangle  \right|}^{2}}+2\left| \left\langle Ax,y \right\rangle  \right|\left| \left\langle Bx,y \right\rangle  \right| \\ 
 &\quad \text{(by the triangle inequality)}\\
 & \le \sqrt{\omega \left( {{\left| A \right|}^{2}}+{\rm i}{{\left| B \right|}^{2}} \right)\omega \left( {{\left| {{A}^{*}} \right|}^{2}}+{\rm i}{{\left| {{B}^{*}} \right|}^{2}} \right)+\omega \left( \left| A \right|+{\rm i}\left| B \right| \right)\omega \left( \left| {{A}^{*}} \right|+{\rm i}\left| {{B}^{*}} \right| \right)\omega \left( B{{A}^{*}} \right)} \\ 
 &\qquad +2\left| \left\langle Ax,y \right\rangle  \right|\left| \left\langle Bx,y \right\rangle  \right| \\ 
 & \le \sqrt{\omega \left( {{\left| A \right|}^{2}}+{\rm i}{{\left| B \right|}^{2}} \right)\omega \left( {{\left| {{A}^{*}} \right|}^{2}}+{\rm i}{{\left| {{B}^{*}} \right|}^{2}} \right)+\omega \left( \left| A \right|+{\rm i}\left| B \right| \right)\omega \left( \left| {{A}^{*}} \right|+{\rm i}\left| {{B}^{*}} \right| \right)\omega \left( B{{A}^{*}} \right)} \\ 
  &\qquad +2\sqrt{\left\langle \left| A \right|x,x \right\rangle \left\langle \left| {{A}^{*}} \right|y,y \right\rangle }\sqrt{\left\langle \left| B \right|x,x \right\rangle \left\langle \left| {{B}^{*}} \right|y,y \right\rangle } \\
   &\quad \text{(by the mixed Schwarz inequality)}\\ 
  & \le \sqrt{\omega \left( {{\left| A \right|}^{2}}+{\rm i}{{\left| B \right|}^{2}} \right)\omega \left( {{\left| {{A}^{*}} \right|}^{2}}+{\rm i}{{\left| {{B}^{*}} \right|}^{2}} \right)+\omega \left( \left| A \right|+{\rm i}\left| B \right| \right)\omega \left( \left| {{A}^{*}} \right|+{\rm i}\left| {{B}^{*}} \right| \right)\omega \left( B{{A}^{*}} \right)} \\ 
  &\qquad +\left\langle \left| A \right|x,x \right\rangle \left\langle \left| {{A}^{*}} \right|y,y \right\rangle +\left\langle \left| B \right|x,x \right\rangle \left\langle \left| {{B}^{*}} \right|y,y \right\rangle \\ 
    &\quad \text{(by the arithmetic-geometric mean inequality)}\\ 
 & \le \sqrt{\omega \left( {{\left| A \right|}^{2}}+{\rm i}{{\left| B \right|}^{2}} \right)\omega \left( {{\left| {{A}^{*}} \right|}^{2}}+{\rm i}{{\left| {{B}^{*}} \right|}^{2}} \right)+\omega \left( \left| A \right|+{\rm i}\left| B \right| \right)\omega \left( \left| {{A}^{*}} \right|+{\rm i}\left| {{B}^{*}} \right| \right)\omega \left( B{{A}^{*}} \right)} \\ 
 &\qquad +\sqrt{\left( {{\left\langle \left| A \right|x,x \right\rangle }^{2}}+{{\left\langle \left| B \right|x,x \right\rangle }^{2}} \right)\left( {{\left\langle \left| {{A}^{*}} \right|y,y \right\rangle }^{2}}+{{\left\langle \left| {{B}^{*}} \right|y,y \right\rangle }^{2}} \right)}\\ 
  &\quad \text{(by the Cauchy-Schwarz inequality)}\\ 
 & \le \sqrt{\omega \left( {{\left| A \right|}^{2}}+{\rm i}{{\left| B \right|}^{2}} \right)\omega \left( {{\left| {{A}^{*}} \right|}^{2}}+{\rm i}{{\left| {{B}^{*}} \right|}^{2}} \right)+\omega \left( \left| A \right|+{\rm i}\left| B \right| \right)\omega \left( \left| {{A}^{*}} \right|+{\rm i}\left| {{B}^{*}} \right| \right)\omega \left( B{{A}^{*}} \right)} \\ 
 &\qquad +\sqrt{\left\langle \left( {{\left| A \right|}^{2}}+{{\left| B \right|}^{2}} \right)x,x \right\rangle \left\langle \left( {{\left| {{A}^{*}} \right|}^{2}}+{{\left| {{B}^{*}} \right|}^{2}} \right)y,y \right\rangle } \\ 
  &\quad \text{(by the Cauchy-Schwarz inequality)}\\ 
  & \le \sqrt{\omega \left( {{\left| A \right|}^{2}}+{\rm i}{{\left| B \right|}^{2}} \right)\omega \left( {{\left| {{A}^{*}} \right|}^{2}}+{\rm i}{{\left| {{B}^{*}} \right|}^{2}} \right)+\omega \left( \left| A \right|+{\rm i}\left| B \right| \right)\omega \left( \left| {{A}^{*}} \right|+{\rm i}\left| {{B}^{*}} \right| \right)\omega \left( B{{A}^{*}} \right)} \\ 
  &\qquad +\sqrt{\left\| \left| A \right|^2+\left| B \right|^2 \right\|\left\|\; \left| {{A}^{*}} \right|^2+\left| {{B}^{*}} \right|^2\right\|}. 
\end{aligned}\]
That is,
\[\begin{aligned}
  & {{\left| \left\langle \left( A+B \right)x,y \right\rangle  \right|}^{2}} \\ 
 & \le \sqrt{\omega \left( {{\left| A \right|}^{2}}+{\rm i}{{\left| B \right|}^{2}} \right)\omega \left( {{\left| {{A}^{*}} \right|}^{2}}+{\rm i}{{\left| {{B}^{*}} \right|}^{2}} \right)+\omega \left( \left| A \right|+{\rm i}\left| B \right| \right)\omega \left( \left| {{A}^{*}} \right|+{\rm i}\left| {{B}^{*}} \right| \right)\omega \left( B{{A}^{*}} \right)} \\ 
 &\qquad +\sqrt{\left\| \left| A \right|^2+\left| B \right|^2 \right\|\left\|\; \left| {{A}^{*}} \right|^2+\left| {{B}^{*}} \right|^2 \right\|}, 
\end{aligned}\]
which implies
\[\begin{aligned}
   {{\left\| A+B \right\|}^{2}}&\le \sqrt{\omega \left( {{\left| A \right|}^{2}}+{\rm i}{{\left| B \right|}^{2}} \right)\omega \left( {{\left| {{A}^{*}} \right|}^{2}}+{\rm i}{{\left| {{B}^{*}} \right|}^{2}} \right)+\omega \left( \left| A \right|+{\rm i}\left| B \right| \right)\omega \left( \left| {{A}^{*}} \right|+{\rm i}\left| {{B}^{*}} \right| \right)\omega \left( B{{A}^{*}} \right)} \\ 
 &\qquad +\sqrt{\left\| \left| A \right|^2+\left| B \right|^2 \right\|\left\| \left| {{A}^{*}} \right|^2+\left| {{B}^{*}} \right|^2\right\|}.  
\end{aligned}\]
Now, if we replace $B$ by ${{e}^{{\rm i}\theta }}B$, we obtain
{\small
\[\begin{aligned}
   \frac{1}{4}{{\left\| A+{{e}^{{\rm i}\theta }}B \right\|}^{2}}&\le \frac{1}{4}\sqrt{\omega \left( {{\left| A \right|}^{2}}+{\rm i}{{\left| B \right|}^{2}} \right)\omega \left( {{\left| {{A}^{*}} \right|}^{2}}+{\rm i}{{\left| {{B}^{*}} \right|}^{2}} \right)+\omega \left( \left| A \right|+{\rm i}\left| B \right| \right)\omega \left( \left| {{A}^{*}} \right|+{\rm i}\left| {{B}^{*}} \right| \right)\omega \left( B{{A}^{*}} \right)} \\ 
 &\qquad +\frac{1}{4}\sqrt{\left\| \left| A \right|^2+\left| B \right|^2 \right\|\left\| \left| {{A}^{*}} \right|^2+\left| {{B}^{*}} \right|^2 \right\|}.  
\end{aligned}\]
}
Now taking supremum over $\theta \in \mathbb{R}$, we conclude
{\small
\[\begin{aligned}
   {{\omega }^{2}}\left( \left[ \begin{matrix}
   O & A  \\
   {{B}^{*}} & O  \\
\end{matrix} \right] \right)&\le \frac{1}{4}\sqrt{\omega \left( {{\left| A \right|}^{2}}+{\rm i}{{\left| B \right|}^{2}} \right)\omega \left( {{\left| {{A}^{*}} \right|}^{2}}+{\rm i}{{\left| {{B}^{*}} \right|}^{2}} \right)+\omega \left( \left| A \right|+{\rm i}\left| B \right| \right)\omega \left( \left| {{A}^{*}} \right|+{\rm i}\left| {{B}^{*}} \right| \right)\omega \left( B{{A}^{*}} \right)} \\ 
 &\qquad +\frac{1}{4}\sqrt{\left\| \left| A \right|^2+\left| B \right|^2 \right\|\left\| \left| {{A}^{*}} \right|^2+\left| {{B}^{*}} \right|^2 \right\|},  
\end{aligned}\]
}
as required.
\end{proof}

If we set $A=T$ and $B={{T}^{*}}$, in Corollary \ref{17}, we get the following.
\begin{corollary}\label{21}
Let $T\in \mathbb B\left( \mathbb H \right)$. Then
\[{{\omega }^{2}}\left( T \right)\le \frac{1}{4}\left( \sqrt{{{\omega }^{2}}\left( {{\left| T \right|}^{2}}+{\rm i}{{\left| {{T}^{*}} \right|}^{2}} \right)+{{\omega }^{2}}\left( \left| T \right|+{\rm i}\left| {{T}^{*}} \right| \right)\omega \left( {{T}^{2}} \right)}+\left\| {{\left| T \right|}^{2}}+{{\left| {{T}^{*}} \right|}^{2}} \right\| \right).\]
\end{corollary}

The inequality in Corollary \ref{21} is sharp. Indeed, if we assume that $T$ is a normal operator, we get the same quantity $\left\| T \right\|$ on both sides of the inequality.

\begin{corollary}\label{22}
Let $T\in \mathbb B\left( \mathbb H \right)$. Then
\[{{\omega }^{2}}\left( T \right)\le \left\| \left( \Re T,\Im T \right) \right\|_{e}^{2}\le {{\omega }^{2}}\left( \left| \Re T \right|+{\rm i}\left| \Im T \right| \right)\le \frac{1}{2}\left\| {{T}^{*}}T+T{{T}^{*}} \right\|.\]
\end{corollary}
\begin{proof}
Very recently, it has been shown in \cite[Theorem 2.3]{5} that if $A,B\in \mathbb B\left( \mathbb H \right)$, then
\begin{equation}\label{24}
{{\left\| \left( A,B \right) \right\|}_{e}}\le \sqrt{\omega \left( \left| A \right|+{\rm i}\left| B \right| \right)\omega \left( \left| {{A}^{*}} \right|+{\rm i}\left| {{B}^{*}} \right| \right)}.
\end{equation}
On the other hand, we know that
	\[{{\omega }_{e}}\left( A,B \right)\le {{\left\| \left( A,B \right) \right\|}_{e}}.\]
Thus,
\begin{equation}\label{23}
{{\omega }_{e}}\left( A,B \right)\le {{\left\| \left( A,B \right) \right\|}_{e}}\le \sqrt{\omega \left( \left| A \right|+{\rm i}\left| B \right| \right)\omega \left( \left| {{A}^{*}} \right|+{\rm i}\left| {{B}^{*}} \right| \right)}.
\end{equation}
Assume that $T=\Re T+{\rm i}\Im T$ be the Cartesian decomposition of $T$. If we replace $A$ and $B$ by $\Re T$ and $\Im T$, and use the fact that
	\[\omega \left( T \right)={{\omega }_{e}}\left( \Re T,\Im T \right),\]
we infer that
\[\begin{aligned}
   {{\omega }^{2}}\left( T \right)&\le \left\| \left( \Re T,\Im T \right) \right\|_{e}^{2} \\ 
 & \le {{\omega }^{2}}\left( \left| \Re T \right|+{\rm i}\left| \Im T \right| \right) \\ 
 & \le \left\| {{\left| \Re T \right|}^{2}}+{{\left| \Im T \right|}^{2}} \right\| \\ 
 & =\left\| {{\left( \Re T \right)}^{2}}+{{\left( \Im T \right)}^{2}} \right\| \\ 
 & =\frac{1}{2}\left\| {{T}^{*}}T+T{{T}^{*}} \right\|, 
\end{aligned}\]
as required.
\end{proof}

As a direct consequence of the first and the second inequality in Corollary \ref{22}, we have the following interesting result.
\begin{corollary}
If $T\in \mathbb B\left( \mathbb H \right)$ is an accretive-dissipative operator, then ${{\left\| \left( \Re T,\Im T \right) \right\|}_{e}}=\omega \left( T \right)$.
\end{corollary}

\begin{corollary}
If $T\in \mathbb B\left( \mathbb H \right)$ is a normal operator, then ${{\left\| \left( \Re T,\Im T \right) \right\|}_{e}}=\omega \left( \left| \Re T \right|+{\rm i}\left| \Im T \right| \right)=\left\| T \right\|$.
\end{corollary}

\begin{remark}
\hfill
\begin{itemize}
\item[(i)] From \eqref{23}, we have
	\[\sqrt{2}\omega \left( T \right)={{\omega }_{e}}\left( T,{{T}^{*}} \right)\le {{\left\| \left( T,{{T}^{*}} \right) \right\|}_{e}}\le \omega \left( \left| T \right|+{\rm i}\left| {{T}^{*}} \right| \right),\]
i.e.,
	\[\omega \left( T \right)\le \frac{\sqrt{2}}{2}{{\left\| \left( T,{{T}^{*}} \right) \right\|}_{e}}\le \frac{\sqrt{2}}{2}\omega \left( \left| T \right|+{\rm i}\left| {{T}^{*}} \right| \right).\]
Notice that in \cite[Corollary 2.2]{4}, it is proved that
	\[\omega \left( T \right)\le \frac{\sqrt{2}}{2}\omega \left( \left| T \right|+{\rm i}\left| {{T}^{*}} \right| \right).\]
	Thus, we have shown a refinement of this inequality in terms of $\|(\cdot,\cdot)\|_e.$
\item[(ii)] We have the following chain of inequalities
\[\begin{aligned}
   \frac{1}{2}\left\| T \right\|&\le \omega \left( \left[ \begin{matrix}
   O & \Re T  \\
   \Im T & O  \\
\end{matrix} \right] \right) \\ 
 & \le \frac{\sqrt{2}}{2}{{\left\| \left( \Re T,\Im T \right) \right\|}_{e}} \\ 
 & \le \frac{\sqrt{2}}{2}\omega \left( \left| \Re T \right|+{\rm i}\left| \Re T \right| \right) \\ 
 & \le \frac{\sqrt{2}}{2}{{\left\| {{\left( \Re T \right)}^{2}}+{{\left( \Im T \right)}^{2}} \right\|}^{\frac{1}{2}}} \\ 
 & \le \frac{\sqrt{2}}{2}\sqrt{{{\left\| \Re T \right\|}^{2}}+{{\left\| \Im T \right\|}^{2}}} \\ 
 & \le \omega \left( T \right).  
\end{aligned}\]
To prove this, we recall the following result, which has been shown recently in \cite[Theorem 2.1]{5} 
	\[\omega \left( \left[ \begin{matrix}
   O & A  \\
   {{B}^{*}} & O  \\
\end{matrix} \right] \right)\le \frac{\sqrt{2}}{2}{{\left\| \left( A,B \right) \right\|}_{e}}.\]
So, from \eqref{24}, we infer that
	\[\omega \left( \left[ \begin{matrix}
   O & A  \\
   {{B}^{*}} & O  \\
\end{matrix} \right] \right)\le \frac{\sqrt{2}}{2}{{\left\| \left( A,B \right) \right\|}_{e}}\le \frac{\sqrt{2}}{2}\sqrt{\omega \left( \left| A \right|+{\rm i}\left| B \right| \right)\omega \left( \left| {{A}^{*}} \right|+{\rm i}\left| {{B}^{*}} \right| \right)}.\]
Assume that $T=\Re T+{\rm i}\Im T$ is the Cartesian decomposition of $T$. 
Replace $A$ and $B$ by $\Re T$ and $\Im T$, in the above inequality, we get
\[\begin{aligned}
   \frac{1}{2}\left\| T \right\|&=\frac{1}{2}\left\| \Re T+{\rm i}\Im T \right\| \\ 
 & \le \omega \left( \left[ \begin{matrix}
   O & \Re T  \\
   {\rm i}\Im T & O  \\
\end{matrix} \right] \right) \quad \text{(by Lemma \ref{lem_hirz})}\\ 
 & =\omega \left( \left[ \begin{matrix}
   O & \Re T  \\
   \Im T & O  \\
\end{matrix} \right] \right) \\ 
 & \le \frac{\sqrt{2}}{2}{{\left\| \left( \Re T,\Im T \right) \right\|}_{e}} \\ 
 & \le \frac{\sqrt{2}}{2}\omega \left( \left| \Re T \right|+{\rm i}\left| \Im T \right| \right).  
\end{aligned}\]
On the other hand,
\[\begin{aligned}
   \omega \left( \left| \Re T \right|+{\rm i}\left| \Im T \right| \right)&\le {{\left\| {{\left( \Re T \right)}^{2}}+{{\left( \Im T \right)}^{2}} \right\|}^{\frac{1}{2}}} \quad \text{(by \cite[Proposition 1.4]{4})}\\ 
 & \le \sqrt{{{\left\| \Re T \right\|}^{2}}+{{\left\| \Im T \right\|}^{2}}} \\ 
 &\quad \text{(by the triangle inequality for the usual operator norm)}\\
 & \le \sqrt{2}\omega \left( T \right).  
\end{aligned}\]
Combining the last two relations implies the desired chain of inequalities.
\item[(iii)] Of course, if $T$ is an accretive-dissipative operator, one can write
\[\begin{aligned}
   \frac{\sqrt{2}}{2}\left\| T \right\|&\le \sqrt{2}\omega \left( \left[ \begin{matrix}
   O & \Re T  \\
   \Im T & O  \\
\end{matrix} \right] \right) \\ 
 & \le {{\left\| \left( \Re T,\Im T \right) \right\|}_{e}} \\ 
 & \le \omega \left( T \right).  
\end{aligned}\]
This provides a refinement of the inequality $\|T\|\leq\sqrt{2}\omega(T)$, valid for accretive-dissipative operators, as shown in \cite{4}.
\end{itemize}
\end{remark}

We conclude with the following bound of $\omega \left( \left[ \begin{matrix}
   O & A  \\
   {{B}^{*}} & O  \\
\end{matrix} \right] \right)$ in terms of $\|(A,B)\|_e.$
\begin{theorem}
Let $A,B\in \mathbb B\left( \mathbb H \right)$. Then
\[\omega \left( \left[ \begin{matrix}
   O & A  \\
   {{B}^{*}} & O  \\
\end{matrix} \right] \right)\le {{\left\| \left( A,B \right) \right\|}_{e}}-\frac{\left| \;\left\| A \right\|-\left\| B \right\| \;\right|}{2}.\]
\end{theorem}
\begin{proof}
Let $x,y \in \mathbb H$ be two unit vectors. One can easily see that
	\[\sqrt{{{\left| \left\langle Ax,y \right\rangle  \right|}^{2}}+{{\left| \left\langle Bx,y \right\rangle  \right|}^{2}}}\ge \left| \left\langle Ax,y \right\rangle  \right|,\sqrt{{{\left| \left\langle Ax,y \right\rangle  \right|}^{2}}+{{\left| \left\langle Bx,y \right\rangle  \right|}^{2}}}\ge \left| \left\langle Bx,y \right\rangle  \right|\]
which implies,
	\[\max \left\{ \left\| A \right\|,\left\| B \right\| \right\}\le {{\left\| \left( A,B \right) \right\|}_{e}}.\]
Notice that
	\[\begin{aligned}
   \omega \left( \left[ \begin{matrix}
   O & A  \\
   {{B}^{*}} & O  \\
\end{matrix} \right] \right)+\frac{\left| \left\| A \right\|-\left\| B \right\| \right|}{2}&\le \frac{1}{2}\left( \left\| A \right\|+\left\| B \right\| \right)+\frac{\left| \;\left\| A \right\|-\left\| B \right\| \;\right|}{2} \\ 
 & =\max \left\{ \left\| A \right\|,\left\| B \right\| \right\}.  
\end{aligned}\]
Combining the last two inequalities implies the desired result.
\end{proof}

\subsection*{Declarations}
\begin{itemize}
\item {\bf{Availability of data and materials}}: Not applicable.
\item {\bf{Conflict of interest}}: The authors declare that they have no conflict of interest.
\item {\bf{Funding}}: Not applicable.
\item {\bf{Authors' contributions}}: Authors declare that they have contributed equally to this paper. All authors have read and approved this version.
\item {\bf{Acknowledgments}}: Not applicable.
\end{itemize}

\vskip 0.3 true cm 	
{\tiny (M. Sababheh) Department of Basic Sciences, Princess Sumaya University for Technology, Amman, Jordan}
	
{\tiny\textit{E-mail address:} sababheh@psut.edu.jo; sababheh@yahoo.com}

\vskip 0.3 true cm 	
{\tiny (H. R. Moradi) Department of Mathematics, Mashhad Branch, Islamic Azad University, Mashhad, Iran
	
\textit{E-mail address:} hrmoradi@mshdiau.ac.ir}

\begin{thebibliography}{99}

\bibitem{Alomari_Georgian_2023}
M. W. Alomari, M. Sababheh, C. Conde, and H. R. Moradi, {\it Generalized Euclidean operator radius}, Georgian Math. J., in press.

\bibitem{2}
S. S. Dragomir, {\it Buzano's inequality holds for any projection}, Bull. Aust. Math. Soc., 93 (2016), 504--510.

\bibitem{10}
S. S. Dragomir, {\it The hypo-Euclidean norm of an n-tuple of vectors in inner product spaces and applications}, J. Inequal. Pure Appl. Math., 8(2) (2007), Article 52, 22.

\bibitem{Dragomir_LAA_2006}
S. S. Dragomir, {\it Some inequalities for the Euclidean operator radius of two operators in Hilbert spaces}, Linear Algebra Appl., 419 (2006), 256--264.


\bibitem{3}
S. S. Dragomir,  {\it Inequalities for the norm and the numerical radius of linear operators in Hilbert spaces}, Demonstratio Math., 40(2) (2007), 411--417.



\bibitem{Gustafson_Book_1997}
K. E. Gustafson, D. K. M. Rao, {\it Numerical Range}, Springer, New York, 1997.

\bibitem{9}
P . R. Halmos, {\it A Hilbert Space Problem Book}, 2nd ed., Springer, New York, 1982.

\bibitem{8}
O. Hirzallah, F. Kittaneh, and K. Shebrawi, {\it Numerical radius inequalities for certain $2\times2$ operator matrices}, Integr. Equ. Oper. Theory., 71 (2011), 129--147.

\bibitem{1}
F. Kittaneh, {\it A numerical radius inequality and an estimate for the numerical radius of the Frobenius companion matrix}, Stud. Math., 158(1) (2003), 11--17.



\bibitem{7}
F. Kittaneh, {\it Numerical radius inequalities for Hilbert space operators}, Stud. Math., 168(1) (2005), 73--80.


\bibitem{Krein_FAA_1969}
M. Kre\u \i n, {\it Angular localization of the spectrum of a multiplicative integral in a Hilbert space}, Funct. Anal. Appl., {3} (1969), 89--90.

\bibitem{Lin_Intelligencer_2012}
M. Lin, {\it Remarks on Kre\u \i n's inequality}, Math. Intell., {34}(1) (2012), 3--4.


 \bibitem{MFS}
 H. R. Moradi, S. Furuichi, and M. Sababheh, {\it Some operator inequalities via convexity}, Linear Multilinear Algebra., 70(22) (2022), 7740--7752.


\bibitem{4}
H. R. Moradi, M. Sababheh, {\it New estimates for the numerical radius}, Filomat., 35(14) (2021), 4957--4962.


\bibitem{Moslehian_MathScand_2017} 
M. S. Moslehian, M. Sattari, and K. Shebrawi, {\it Extensions of Euclidean operator radius inequalities}, Math. Scand., {120}(1) (2017), 129--144.

\bibitem{11}
G. Popescu, {\it Unitary invariants in multivariable operator theory}, Mem. Amer. Math. Soc., 200 (2009), No. 941, vi+91 pp.

\bibitem{5}
M. Sababheh, H. R. Moradi, and M. Alomari, {\it On the Euclidean operator radius and norm}, 	arXiv:2309.11138 [math.FA]

\bibitem{SMF}
M. Sababheh, H. R. Moradi, and S. Furuichi{\it Operator inequalities via geometric convexity}, Math. Inequal. Appl., 22(4) (2019), 1215--1231.



\bibitem{Sababheh_OaM_2022}
M. Sababheh, H. R. Moradi, and Z. Heydarbeygi, {\it Buzano, Kre\u\i n and Cauchy-Schwarz inequalities}, Oper. Matrices., {16}(1) (2022), 239--250.

 \bibitem{Scharnhorst_ActaApplMath_2001}
 K. Scharnhorst, {\it Angles in complex vector spaces}, Acta Appl. Math., {69} (2001), 95--103.
 
 \bibitem{Yamazaki}
T. Yamazaki, {\it On upper and lower bounds of the numerical radius and an equality condition}, Studia Math., 178 (2007) 83--89.
\end{thebibliography}
\end{document}